\documentclass{amsart}
\usepackage{amscd,amsmath,xypic,amssymb,combelow,tikz-cd,etoolbox,calligra,mathrsfs,stmaryrd}

\DeclareMathOperator{\sHom}{\mathscr{H}\text{\kern -3pt {\calligra\large om}}\,}
\DeclareMathOperator{\sRHom}{\mathscr{RH}\text{\kern -3pt {\calligra\large om}}\,}
\DeclareMathOperator{\sQuot}{\mathscr{Q}\text{\kern -3pt {\calligra\large uot}}\,}

\makeatletter
\patchcmd{\@settitle}{\uppercasenonmath\@title}{}{}{}
\makeatother

\xyoption{all}
\CompileMatrices

\newcommand{\nc}{\newcommand}

\makeatletter
\@addtoreset{equation}{section}
\makeatother

\newtheorem{theorem}[subsection]{Theorem}
\newtheorem{proposition}[subsection]{Proposition}
\newtheorem{lemma}[subsection]{Lemma}
\newtheorem{corollary}[subsection]{Corollary}
\newtheorem{conjecture}[subsection]{Conjecture}
\newtheorem{definition}[subsection]{Definition}

\newtheorem{claim}[subsection]{Claim}

\nc{\fa}{{\mathfrak{a}}}
\nc{\fb}{{\mathfrak{b}}}
\nc{\fd}{{\mathfrak{d}}}
\nc{\fg}{{\mathfrak{g}}}
\nc{\fh}{{\mathfrak{h}}}
\nc{\fj}{{\mathfrak{j}}}
\nc{\fn}{{\mathfrak{n}}}
\nc{\fm}{{\mathfrak{m}}}
\nc{\fu}{{\mathfrak{u}}}
\nc{\fp}{{\mathfrak{p}}}
\nc{\fr}{{\mathfrak{r}}}
\nc{\ft}{{\mathfrak{t}}}
\nc{\fsl}{{\mathfrak{sl}}}
\nc{\fgl}{{\mathfrak{gl}}}
\nc{\hsl}{{\widehat{\mathfrak{sl}}}}
\nc{\hgl}{{\widehat{\mathfrak{gl}}}}
\nc{\hg}{{\widehat{\mathfrak{g}}}}
\nc{\chg}{{\widehat{\mathfrak{g}}}{}^\vee}
\nc{\hn}{{\widehat{\mathfrak{n}}}}
\nc{\chn}{{\widehat{\mathfrak{n}}}{}^\vee}

\nc{\Mod}{{\textrm{Mod}}}

\nc{\wGL}{{\widehat{GL}^+}}

\nc{\BA}{{\mathbb{A}}}
\nc{\BC}{{\mathbb{C}}}
\nc{\BG}{{\mathbb{G}}}
\nc{\BM}{{\mathbb{M}}}
\nc{\BN}{{\mathbb{N}}}
\nc{\BF}{{\mathbb{F}}}
\nc{\BH}{{\mathbb{H}}}
\nc{\BP}{{\mathbb{P}}}
\nc{\BQ}{{\mathbb{Q}}}
\nc{\BR}{{\mathbb{R}}}
\nc{\BZ}{{\mathbb{Z}}}

\nc{\ff}{{\mathbb{F}}}
\nc{\kk}{{\mathbb{K}}}
\nc{\kko}{{\mathbb{K}}}

\nc{\coh}{{\text{Coh}}}

\nc{\CA}{{\mathcal{A}}}
\nc{\CC}{{\mathcal{C}}}
\nc{\CB}{{\mathcal{B}}}
\nc{\DD}{{\mathcal{D}}}
\nc{\CE}{{\mathcal{E}}}
\nc{\CF}{{\mathcal{F}}}
\nc{\tCF}{{\widetilde{\CF}}}
\nc{\tCM}{{\widetilde{\CM}}}
\nc{\tCT}{{\widetilde{\CT}}}
\nc{\oCF}{{\bar{\CF}}}
\nc{\CG}{{\mathcal{G}}}
\nc{\CL}{{\mathcal{L}}}
\nc{\CK}{{\mathcal{K}}}
\nc{\CI}{{\mathcal{I}}}
\nc{\CM}{{\mathcal{M}}}
\nc{\CH}{{\mathcal{H}}}
\nc{\CN}{{\mathcal{N}}}
\nc{\CO}{{\mathcal{O}}}
\nc{\CP}{{\mathcal{P}}}
\nc{\CR}{{\mathcal{R}}}
\nc{\CQ}{{\mathcal{Q}}}
\nc{\CS}{{\mathcal{S}}}
\nc{\CT}{{\mathcal{T}}}
\nc{\tCU}{{\widetilde{\CU}}}
\nc{\CU}{{\mathcal{U}}}
\nc{\CV}{{\mathcal{V}}}
\nc{\CW}{{\mathcal{W}}}
\nc{\CZ}{{\mathcal{Z}}}

\nc{\tpsi}{{\widetilde{\Psi}}}
\nc{\wpi}{{\widetilde{\pi}}}
\nc{\Ker}{{\text{Ker }}}
\nc{\Coker}{{\text{Coker }}}

\nc{\CX}{{\mathcal{X}}}
\nc{\tCX}{{\widetilde{\mathcal{X}}}}
\nc{\CY}{{\mathcal{Y}}}
\nc{\tCY}{{\widetilde{\mathcal{Y}}}}

\nc{\tN}{{\widetilde{\CN}}}
\nc{\pN}{{\BP\widetilde{\CN}}}

\nc{\tT}{{T}}

\nc{\fq}{{\mathfrak{q}}}

\nc{\fC}{{\mathfrak{C}}}
\nc{\fJ}{{\mathfrak{J}}}
\nc{\fG}{{\mathfrak{G}}}
\nc{\fL}{{\mathfrak{L}}}
\nc{\fZ}{{\mathfrak{Z}}}
\nc{\fU}{{\mathfrak{U}}}
\nc{\fV}{{\mathfrak{V}}}
\nc{\fS}{{\mathfrak{S}}}

\nc{\od}{{\overline{d}}}
\nc{\rg}{{\textrm{R}\Gamma}}
\nc{\erg}{{\emph{R}\Gamma}}
\nc{\id}{{\textrm{Id}}}
\nc{\rhom}{{\textrm{RHom}}}

\def\Ext{\textrm{Ext}}
\def\Hom{\textrm{Hom}}
\def\RHom{\textrm{RHom}}
\def\e{\varepsilon}

\def\tCF{\widetilde{\CF}}

\def\fsp{\mathfrak{sp}}
\def\fgl{\mathfrak{gl}}

\def\fb{\mathfrak{b}}
\def\fh{\mathfrak{h}}
\def\fg{\mathfrak{g}}
\def\fp{\mathfrak{p}}
\def\ft{\mathfrak{t}}

\def\loccit{\emph{loc. cit. }}
\def\loccitt{\emph{loc. cit.}}

\def\Id{\text{Id}}

\def\sgn{\text{sign }}

\def\Tan{{\text{Tan}}}
\def\eTan{{\emph{Tan}}}
\def\Hilb{{\text{Hilb}}}
\def\eHilb{{\emph{Hilb}}}

\def\ch{{\text{ch}}}
\def\ech{{\emph{ch}}}
\def\td{{\text{td}}}

\def\trans{{\text{tr}}}

\def\Pic{{\text{Pic}}}

\def\heis{{\text{Heis}}}
\def\vir{{\text{Vir}}}

\def\tr{{\text{tr}}}
\def\sm{{\text{small}}}
\def\big{{\text{big}}}

\def\as{{A^*(S)}}
\def\ah{{A^*(\Hilb)}}
\def\ahs{{A^*(\Hilb \times S)}}

\def\eah{{A^*(\eHilb)}}
\def\eahs{{A^*(\eHilb \times S)}}
\def\eahss{{A^*(\eHilb \times S^2)}}

\begin{document}

\title[Lehn's formula in Chow and conjectures of Beauville and Voisin]{\large{\textbf{LEHN'S FORMULA IN CHOW AND CONJECTURES OF BEAUVILLE AND VOISIN}}}
\author[Davesh Maulik, Andrei Negu\cb t]{Davesh Maulik, Andrei Negu\cb t}

\address{D.M. Massachusetts Institute of Technology, Department of Mathematics, 77 Massachusetts Avenue, Cambridge, MA 02139, US}
\email{maulik@mit.edu}

\address{A.N. Massachusetts Institute of Technology, Department of Mathematics, 77 Massachusetts Avenue, Cambridge, MA 02139, US}
\address{Simion Stoilow Institute of Mathematics, Bucharest, Romania}
\email{andrei.negut@gmail.com}

%\renewcommand{\thefootnote}{\fnsymbol{footnote}} 
%\footnotetext{\emph{2010 Mathematics Subject Classification: }Primary 14J60, Secondary 14D21}     
%\renewcommand{\thefootnote}{\arabic{footnote}} 

\maketitle

\begin{abstract} The Beauville-Voisin conjecture for a hyperk\"ahler manifold $X$ states that the subring of the Chow ring $A^*(X)$ generated by divisor classes and Chern characters of the tangent bundle injects into the cohomology ring of $X$. We prove a weak version of this conjecture when $X$ is the Hilbert scheme of points on a K3 surface, for the subring generated by divisor classes and tautological classes.  This in particular implies the weak splitting conjecture of Beauville for these geometries.  In the process, we extend Lehn's formula and the Li-Qin-Wang $W_{1+\infty}$ algebra action from cohomology to Chow groups, for the Hilbert scheme of an arbitrary smooth projective surface $S$.

\end{abstract}

\section{Introduction}

\subsection{} \label{sub:intro main} 

We will work with smooth algebraic varieties $X$ over an algebraically closed field of characteristic 0, henceforth denoted by $\BC$. For such a variety $X$, we will write $A^*(X)$ and $H^*(X)$ for its Chow group and even-degree cohomology group with $\BQ$--coefficients, respectively. When $X = \Hilb_n(S)$ is the Hilbert scheme of $n$ points on a K3 surface $S$, a significant source of elements of $A^*(X)$ is given by the universal subscheme:
\begin{equation}
\label{eqn:first diag}
\xymatrix{\CZ_n \ar@{^{(}->}[r] & \Hilb_n(S) \times S \ar[d]_{\pi_1} \ar[rd]^{\pi_2} & \\ & \Hilb_n(S) & S}
\end{equation}
We define a small tautological class to be any element of $A^*(\Hilb_n(S))$ of the form:
\begin{equation}
\label{eqn:intro small}
\pi_{1*} \Big[ \ch_{k}(\CO_{\CZ_n}) \cdot \pi_2^*(\gamma) \Big] \qquad \forall k \in \BN, \gamma \in R(S)
\end{equation}
where $R(S) \subset A^*(S)$ is the subring generated by divisor classes. Our main result is: \\

\begin{theorem}
\label{thm:main}
	
The cycle class map $A^*(\eHilb_n(S)) \rightarrow H^*(\eHilb_n(S))$ is injective on the subring generated by small tautological classes, for any K3 surface $S$ and $n \in \BN$. \\
	
\end{theorem}

\noindent This result is motivated by the following conjecture of Beauville and Voisin (\cite{V}): \\

\begin{conjecture}
\label{conj:bv}

For any hyperk\"ahler $X$, the cycle class map $A^*(X) \rightarrow H^*(X)$ is injective on the subring generated by divisor classes and Chern classes of $T_X$. \\

\end{conjecture} 

\noindent Let us henceforth specialize to $X = \Hilb_n(S)$ for a K3 surface $S$ and any $n \in \BN$. Then our Theorem \ref{thm:main} implies the weak splitting conjecture (\cite{B2}) of Beauville for arbitrary $n$ (because of Proposition \ref{prop:divisor}).  The latter is a weaker version of Conjecture \ref{conj:bv}, where one only considers the subring of $A^*(X)$ generated by divisor classes. \\

\noindent Voisin proved Conjecture \ref{conj:bv} for $\Hilb_n(S)$ with $n \leq 2b+1$, where $b$ is the rank of the transcendental lattice of $S$. The upper bound on $n$ was improved to $(b+1)(b+2)$ by \cite{FT}, in relation with other conjectures on Chow groups of algebraic varieties (where they also proved the weak splitting property for $n < 506$).  Yin (\cite{Y}) showed Conjecture \ref{conj:bv} and a related conjecture  of Voisin \ref{conj:v} hold for all $n$ when the surface $S$ has a finite-dimensional motive in the sense of \cite{K}, which is known to hold in several examples.
There is much ongoing work of Ayoub on establishing the latter finite-dimensionality statement for all varieties, which would prove all the conjectures mentioned in the present paper. \\

\subsection{} \label{sub:philosophy}

Our approach to proving Theorem \ref{thm:main} is package relations in Chow in the language of representation theory. In short, we consider the Lie algebra: 
\begin{equation}
\label{eqn:virasoro}
\heis \times \vir
\end{equation}
where $\heis$ denotes a rank $(24-b)$ infinite-dimensional Heisenberg algebra and $\vir$ denotes the Virasoro algebra. There is an action of the Lie algebra above on:
$$
A^*(\Hilb) = \bigoplus_{n=0}^\infty A^*(\Hilb_n(S))
$$
which lifts the well-known action on cohomology.
We show that this action preserves the subring $V_\sm \subset A^*(\Hilb)$ generated by small tautological classes \eqref{eqn:intro small}. Furthermore, we show that $V_\sm$ is generated under the $\heis \times \vir$ action by $A^*(\Hilb_0(S)) \cong \BQ$, thus forming a lowest-weight module. The classification of lowest weight modules of $\heis$ (which is trivial) and of $\vir$ (which was developed in \cite{FF}) allows us to conclude that $V_\sm$ is (almost) an irreducible representation of $\heis \times \vir$. Therefore, Schur's lemma implies that $V_\sm$ injects into cohomology, thus establishing Theorem \ref{thm:main}. \\

\subsection{} \label{sub:intro lehn} The key ingredient in the above argument is that any product of small tautological classes can be obtained from $A^*(\Hilb_0(S)) \cong \BQ$ under the action of the Lie algebra \eqref{eqn:virasoro}. The analogous statement in cohomology follows from certain important results in geometric representation theory, namely Lehn's formula (\cite{L}) and the Li-Qin-Wang $W_{1+\infty}$ algebra action (\cite{LQW}). Therefore, most of the technical work that goes into the present paper is to lift the aforementioned results from cohomology to Chow rings. In more detail, recall the following operators studied by Nakajima (\cite{Nak}), and in a different formulation, by Grojnowski (\cite{G}):
\begin{equation}
\label{eqn:nak 1}
A^*(\Hilb) \xrightarrow{\fq_n} A^*(\Hilb \times S), \qquad A^*(\Hilb) \xrightarrow{\fq_n(\gamma)} A^*(\Hilb)
\end{equation}
defined for all $n \in \BZ \backslash 0$ and $\gamma \in A^*(S)$ by formulas \eqref{eqn:nak exp 1}, \eqref{eqn:nak exp 2}, \eqref{eqn:nakajima class}. The operators \eqref{eqn:nak 1} satisfy the relations of the Heisenberg algebra (\cite{G}, \cite{Nak}):
$$
[\fq_n(\gamma), \fq_{n'}(\gamma')] = n \delta_{n+n'}^0 \langle \gamma, \gamma' \rangle \cdot \text{Id}_\Hilb
$$
where $\langle, \rangle$ denotes the intersection pairing on $A^*(\Hilb)$. One may also adapt the notation above to compositions of several operators \eqref{eqn:nak 1}, for any $n_1,...,n_k \in \BZ \backslash 0$:
$$
A^*(\Hilb) \xrightarrow{\fq_{n_1}...\fq_{n_k}} A^*(\Hilb \times S^k), \qquad A^*(\Hilb) \xrightarrow{\fq_{n_1}...\fq_{n_k}(\Gamma)} A^*(\Hilb)
$$
and any $\Gamma \in A^*(S^k)$. A particular instance of this construction is given by the following Virasoro operators, which are close relatives of the operators constructed by Lehn (\cite{L}) in cohomology:
$$
L_n : A^*(\Hilb) \rightarrow A^*(\Hilb), \qquad L_n = \frac 12 \sum_{a+b = n} : \fq_a \fq_b: (\Delta^\tr)
$$
where $\Delta^\tr \in A^*(S \times S)$ is defined in \eqref{eqn:transcendental} and $: \ :$ denotes normal-ordering (see \eqref{eqn:normal}). Then our strategy for proving Theorem \ref{thm:main} is to consider the Lie algebra:
\begin{equation}
\label{eqn:the algebra}
\heis \times \vir = \Big \{ \fq_n(\gamma), L_{n'} \Big\}_{n \in \BZ \backslash 0, n' \in \BZ}^{\gamma \in R(S)} 
\end{equation}
which acts on $A^*(\Hilb)$ as explained above. To prove Theorem \ref{thm:main}, we must show that this action preserves the subring of $A^*(\Hilb)$ generated by small tautological classes, and that moreover it generates the latter subring from $A^*(\Hilb_0(S)) \cong \BQ$. To show this, we prove the following Chow-theoretic version of Lehn's formula (an equivalent version of equation (1) of \cite{L}) for any smooth projective surface $S$: \\

\begin{theorem}
\label{thm:lehn}

We have the equalities of operators $A^*(\eHilb) \rightarrow A^*(\eHilb \times S)$:
\begin{equation}
\label{eqn:degree}
\fG_2 = - \sum_{n = 1}^\infty \fq_{n} \fq_{-n} \Big|_\Delta 
\end{equation}
\begin{equation}
\label{eqn:lehn}
\fG_3 = - \frac 16 \sum_{n_1+n_2+n_3 = 0} : \fq_{n_1} \fq_{n_2} \fq_{n_3}: \Big|_\Delta - \frac t2 \sum_{n=1}^{\infty} n \fq_n \fq_{-n} \Big|_\Delta
\end{equation}
where $\fG_k : A^*(\eHilb) \rightarrow A^*(\eHilb \times S)$ is pullback followed by multiplication with $\ech_k(\CO_{\CZ})$, where $\CZ \subset \eHilb \times S$ is the universal subscheme \eqref{eqn:subscheme} and $t = c_1(\CK_S)$. \\

\end{theorem}

\noindent The formulas above hold for any smooth projective surface $S$ and one sets $t = 0$ in the particular case of a K3 surface. Theorem \ref{thm:lehn} was shown in \cite{L} in cohomology; however, the argument given there does not generalize to Chow.
Indeed, the proof in cohomology relies critically on the fact that cohomology of Hilbert schemes form an irreducible module for the Heisenberg algebra.  This reduces the identity to showing both sides have the same commutation relations with the Nakajima operators \eqref{eqn:nak 1}.  This approach breaks down for Chow groups, which are too large to form an irreducible module of the Heisenberg algebra.
Instead, we will prove Theorem \ref{thm:lehn} by a more intersection-theoretic argument (which also leads to a new proof of Lehn's formula in cohomology) in Section \ref{sec:lehn}. \\

\noindent In cohomology, the study of the operators $\fG_k$ was systematized by Li-Qin-Wang in \cite{LQW}, where the authors showed that the algebra generated by $\fG_k$ and $\fq_n$ satisfies the relations in the deformed $W_{1+\infty}$ algebra. To prove this statement, \loccit also use the irreducibility of $H^*(\Hilb)$ as a module over the Heisenberg algebra. In Section \ref{sec:rep theory}, we will prove that the following version of their result also holds in Chow: \\
  
 \begin{theorem}
 \label{thm:lqw}
 	
If $S$ has $c_1 (\eTan_S)= 0$ and $c_2 (\eTan_S) = e$, then there exist operators $\{\fJ_n^k : \eah \rightarrow \eahs\}_{n \in \BZ}^{k \geq 0}$ determined by the following conditions:
\begin{align}
&\fJ_n^0 = - \fq_n \label{eqn:lqw 1} \\
%&\fJ_n^1 = - \fL_n \label{eqn:lqw 2} \\
&\fJ_0^k = k! \left( \fG_{k+1} + \frac {\pi_2^*(e)}{12} \cdot \fG_{k-1} \right) \label{eqn:lqw 3}
\end{align}
and the following relations for all $n,n' \in \BZ$ and $k,k' \geq 0$ with $k+k' \geq 3$:
\begin{align}
&[\fJ_n^k, \fJ_{n'}^{k'}] = (kn' - k'n) \Delta_*(\fJ_{n+n'}^{k+k'-1}) + \Omega_{n,n'}^{k,k'} \Delta_* \left(\frac {\pi_2^*(e)}{12} \cdot  \fJ_{n+n'}^{k+k'-3}\right) \label{eqn:lqw rel 1} \\
&[\fJ_n^0, \fJ_{n'}^0] = n \delta_{n+n'}^0 \Delta_*(\pi_1^*) \label{eqn:lqw rel 2} \\
&[\fJ_n^1, \fJ_{n'}^0] = n' \Delta_*( \fJ_{n+n'}^0) \label{eqn:lqw rel 3} \\
&[\fJ_n^2, \fJ_{n'}^0] = 2n' \Delta_*( \fJ_{n+n'}^1) - \frac {n^3-n}6 \delta_{n+n'}^0 \Delta_*(\pi_2^*(e) \cdot \pi_1^*) \label{eqn:lqw rel 4} \\
&[\fJ_n^1, \fJ_{n'}^1] = (n'-n) \Delta_*( \fJ_{n+n'}^1) - \frac {n^3-n}{12} \delta_{n+n'}^0 \Delta_*(\pi_2^*(e) \cdot \pi_1^*) \label{eqn:lqw rel 5}
\end{align}
(see Theorem 5.5 of \cite{LQW} for the precise formula of the integers $\Omega_{n,n'}^{k,k'}$, and note that our $\fJ_n^k$ are $\fJ_{-n}^k$ of loc. cit.). The two sides of each of relations \eqref{eqn:lqw rel 1}--\eqref{eqn:lqw rel 5} are homomorphisms $\eah \rightarrow \eahss$, with each of the operators $\fJ_n^k$ and $\fJ_{n'}^{k'}$ in the LHS acting in one and the same of the two factors of $S^2$. \\

\end{theorem}

\subsection{}

\noindent As we mentioned, the connection between Theorem \ref{thm:main} and Conjecture \ref{conj:bv} (for $X = \Hilb_n(S))$ is that divisor classes are among the small tautological classes, but the Chern classes of the tangent bundle are not. To understand the latter, one needs to consider instead the set of big tautological classes, namely:
\begin{equation}
\label{eqn:intro big}
\pi_{1*} \Big[ \ch_{k_1}(\CO_{\CZ_n}) ... \ch_{k_t}(\CO_{\CZ_n}) \cdot \pi_2^*(\gamma) \Big] \qquad \forall k_1,...,k_t \in \BN, \gamma \in R(S)
\end{equation}
with the notation in \eqref{eqn:first diag}. Then we propose: \\

\begin{conjecture}
\label{conj:main}
	
The cycle class map $A^*(\eHilb_n(S)) \rightarrow H^*(\eHilb_n(S))$ is injective on the subring generated by big tautological classes, for any $n \in \BN$. \\

\end{conjecture}

\noindent Note that Conjecture \ref{conj:main} implies  Conjecture \ref{conj:bv} for $X = \Hilb_n(S)$, as a consequence of Proposition \ref{prop:chern tangent}. In \cite{V}, Voisin proposed the following: \\

\begin{conjecture}
\label{conj:v}

Let $p_i : S^n \rightarrow S$ denote the $i$-th projection. For any $n \in \BN$, the restriction of the cycle class map $A^*(S^n) \rightarrow H^*(S^n)$ to the subring generated by: 
$$
\Big \{ p_i^*(l) \Big\}^{l \in A^1(S)}_{1 \leq i \leq n} \quad \text{and} \quad \Big\{ (p_i \times p_j)^*(\Delta) \Big\}_{1 \leq i < j \leq n}
$$
is injective. Above, $\Delta$ denotes the class of the diagonal in $A^*(S \times S)$. \\

\end{conjecture}

\noindent In \cite{Y}, Conjecture \ref{conj:v} was shown to boil down to the ``Kimura relation", a formula in the Chow ring of $S^{2(b+1)}$ that we recall in \eqref{eqn:yin} (here $b$ is the rank of the transcendental lattice of $S$). By a standard argument, one has: \\

\begin{proposition}
\label{prop:equiv}

Conjecture \ref{conj:main} is equivalent to Conjecture \ref{conj:v}. \\

\end{proposition}

\subsection{} \label{sub:intro big} One may ask if the representation theoretic approach of Subsection \ref{sub:philosophy} can be generalized to prove the more general Conjecture \ref{conj:main}. The answer is no, since developing such a framework to attack Conjecture \ref{conj:main} will necessarily boil down to the Kimura relation, which was already known (\cite{Y}) to imply Conjecture \ref{conj:v}. In more detail, if one wanted an algebra $\fg$ that acts on $A^*(\Hilb)$ such that all big tautological classes can be generated via $\fg$ from $A^*(\Hilb_0) \cong \BQ$, then one would need to take:
\begin{equation}
\label{eqn:symplectic}
\fg = \heis \times \fsp_{2\infty} = \Big\{ \fq_n(\gamma), \fq_n\fq_m(\Delta^\tr) \Big\}^{\gamma \in R(S)}_{m,n \in \BZ \backslash 0}
\end{equation}
Unfortunately, we will explain in Section \ref{sec:proof} that the representation theory of $\fsp_{2\infty}$ alone is not enough to establish Conjecture \ref{conj:main}. This is because the classification of lowest weight $\fsp_{2\infty}$--modules is more complicated than that of $\vir$--modules, and proving that the subring of $A^*(\Hilb)$ generated by big tautological classes is an irreducible module for $\fsp_{2\infty}$ is at least as hard as proving the Kimura relation \eqref{eqn:yin}. \\

\subsection{} We would like to thank Pavel Etingof, Daniel Huybrechts, and Ivan Losev for many interesting discussions. We are especially grateful to Claire Voisin for several conversations on these topics.  D.M. is partially supported by NSF FRG grant DMS-1159265.  A.N. would like to thank MSRI, Berkeley, for their hospitality while this paper was being written in the Spring of 2018, and gratefully acknowledges the support of NSF grants DMS-1600375 and DMS-1440140. \\

\section{The Chow ring and Hilbert schemes of a K3 surface}

\subsection{} In the present paper, $A^*(X)$ will denote the Chow ring of a smooth projective variety $X$ with coefficients in $\BQ$, with the grading by codimension. In the particular case of a K3 surface $S$, Beauville and Voisin (\cite{BV}) have studied the class $c \in A^2(S)$ of any closed point on a rational curve in $S$. They proved the following relations:
\begin{equation}
\label{eqn:bv 1}
c_2(\Tan_S) = 24c 
\end{equation}
\begin{equation}
\label{eqn:bv 2}
l \cdot l' = \langle l,l'\rangle c
\end{equation}
for all $l,l' \in A^1(S)$. In \eqref{eqn:bv 1}, $\Tan_S$ denotes the tangent bundle of the surface $S$. In \eqref{eqn:bv 2}, we use the notation $\langle \cdot, \cdot \rangle : A^*(S) \otimes A^*(S) \rightarrow \BQ$ for the intersection pairing. Moreover, Beauville and Voisin prove the following equalities in $A^*(S \times S)$, where we will write $l_i,c_i$ for the classes $l,c$ pulled back from the $i$-th factor, $i \in \{1,2\}$:
\begin{equation}
\label{eqn:bv 3}
\Delta \cdot c_1 = \Delta \cdot c_2 = c_1 \cdot c_2
\end{equation}
\begin{equation}
\label{eqn:bv 4}
\Delta \cdot l_1 = \Delta \cdot l_2 = l_1 \cdot c_2 + l_2 \cdot c_1 
\end{equation}
Finally, we have the following formulas in $A^*(S \times S \times S)$, where $\Delta_{ij}$ will denote the class of the codimension 2 diagonal pulled back from the $i$-th and $j$-th factor, and $\Delta_{123} = \Delta_{12}\cdot \Delta_{23}$ denotes the class of the smallest (dimension 2) diagonal:
\begin{equation}
\label{eqn:bv 5}
\Delta_{123} = \Delta_{12} \cdot c_3 + \Delta_{13} \cdot c_2 + \Delta_{23} \cdot c_1 - c_1 \cdot c_2 - c_1 \cdot c_3 - c_2 \cdot c_3
\end{equation}
Combining \eqref{eqn:bv 3} with \eqref{eqn:bv 5}, one obtains the following formula for the class $\Delta_{12...n}$ of the smallest (dimension 2) diagonal inside $S^n$, for any natural number $n$:
\begin{equation}
\label{eqn:bv 6}
\Delta_{12...n} = \sum_{1 \leq i < j \leq n} \Delta_{ij} \prod_{k \neq i,j} c_k - (n-2) \sum_{i=1}^n \prod_{k \neq i} c_k
\end{equation}
Thus, it is a feature of K3 surfaces that arbitrary diagonals in $S^n$ can be expressed in terms of codimension 2 diagonals, and the pull-back of $c$ from the various factors. \\

\subsection{}
\label{sub:bv}

It is convenient to consider the following modification of the diagonal class:
\begin{equation}
\label{eqn:transcendental}
\Delta^\tr = \Delta - c_1 - c_2 - \sum_i l_{(i)1} l^{(i)}_2 \in A^*(S \times S)
\end{equation}
where $\{l_{(i)},l^{(i)}\}$ denote dual bases of $\Pic(S) \otimes \BQ$ with respect to the intersection pairing. The notation reflects the fact that the image of $\Delta^\tr$ in cohomology is the canonical tensor of the transcendental lattice (which is the orthogonal complement of the Picard lattice). We will denote by $b$ the rank of the transcendental lattice:
\begin{equation}
\label{eqn:b}
b = \langle \Delta^\tr, \Delta \rangle
\end{equation}
and note that it is an integer contained between 2 and 21. The classes \eqref{eqn:transcendental} will be useful for us because relations \eqref{eqn:bv 3} and \eqref{eqn:bv 4} can be rewritten as:
\begin{equation}
\label{eqn:bv 7}
\Delta^\tr \cdot l = \Delta^\tr \cdot c = 0
\end{equation}
Let $R(S^n) \subset A^*(S^n)$ denote the subring generated by the diagonal classes $\Delta_{ij}$ and the classes $l_i,c_i$ for all $1 \leq i < j \leq n$, as $l$ goes over $A^1(S)$. Conjecture \ref{conj:v} is a statement about the injectivity of the restriction of the cycle class map to $R(S^n)$. In \cite{Y}, Yin showed that Conjecture \ref{conj:v} is equivalent to the equality:
\begin{equation}
\label{eqn:yin}
\sum_{\sigma \in \Sigma_{b+1}} \sgn (\sigma) \prod_{i=1}^{b+1} \Delta^\tr_{i,\sigma(i)+b+1} = 0 \in A^*(S^{2(b+1)})
\end{equation}
Above, we write $\Sigma_{b+1}$ for the symmetric group on $b+1$ letters. \\

\subsection{} Given a K3 surface $S$, we let $\Hilb_n = \Hilb_n(S)$ denote the Hilbert scheme parametrizing colength $n$ ideals $I \subset \CO_S$. The following result is classical: \\

\begin{proposition}
\label{prop:basic}

The variety $\eHilb_n$ is smooth and projective of dimension $2n$. \\

\end{proposition}

\noindent (The smoothness part of the Proposition above is due to Fogarty). The Hilbert scheme represents the functor of flat families of ideal sheaves, i.e.:
\begin{equation}
\label{eqn:represents}
\text{Maps}(T,\Hilb_n) \cong \Big\{ I \subset \CO_{T \times S} \text{ s.t. } \CO_{T \times S}/I \text{ is locally free of rank }n \text{ on }T \Big\}
\end{equation}
for any scheme $T$. We will use the notation $\CI$ for the universal ideal sheaf:
\begin{equation}
\label{eqn:universal ideal}
\xymatrix{\CI \ar@{..>}[d] \\ \Hilb_n \times S}
\end{equation}
in terms of which the identification \eqref{eqn:represents} is given by:
$$
\left\{ T \xrightarrow{\phi} \Hilb_n \right\} \rightarrow \Big\{ I = (\phi \times \Id_S)^{-1}(\CI) \Big\}
$$
The quotient:
\begin{equation}
\label{eqn:subscheme}
\CO_{\CZ_n} = \CO_{\Hilb_n \times S}/\CI 
\end{equation}
is the structure sheaf of the universal subscheme $\CZ_n \subset \Hilb_n \times S$, namely the codimension 2 subscheme supported on the closed subset of pairs $(I,x)$, where $I \subset \CO_S$ is an ideal and $x$ is a support point of $\CO_S/I$. We will write $\CZ = \sqcup_{n=0}^\infty \CZ_n$. \\

\subsection{} 
\label{sub:taut} 

 Since $\CZ$ is a codimension 2 subscheme, we have:
\begin{align}
&\ch_0 ( \CO_{\CZ} ) = 0 \label{eqn:ch taut 1} \\
&\ch_1 ( \CO_{\CZ} ) = 0 \label{eqn:ch taut 2} \\
&\ch_2 ( \CO_{\CZ} ) = [\CZ] \label{eqn:ch taut 3}
\end{align}
Using the Chern character of $\CO_\CZ$ allows us to define various types of classes in $A^*(\Hilb)$. Recall that $\pi_1, \pi_2 : \Hilb \times S \rightarrow \Hilb, S$ denote the two standard projections, and $R(S) \subset A^*(S)$ denotes the Beauville-Voisin subring:
\begin{equation}
\label{eqn:bv subring}
R(S) = \BQ \cdot 1 \oplus c_1(\Pic(S)) \oplus \BQ \cdot c
\end{equation}

%For any vector bundle $F$ on $S$ of rank $r$, we may define the \textbf{tautological bundle}:
%\begin{equation}
%\label{eqn:tautological}
%F^{[n]} = p_{1*} \left( \frac {\CO_{\Hilb_n \times S}}{\CI} \otimes p_2^*(F) \right)
%\end{equation}
%Because the quotient is flat over $\Hilb_n$, \eqref{eqn:tautological} is a vector bundle on $\Hilb_n$ of rank $rn$. The Grothendieck-Hirzebruch-Riemann-Roch theorem gives us the following formula for $\ch(F^{[n]}) \in A^*(\Hilb_n)$:
%\begin{equation}
%\label{eqn:formula 0}
%\ch (F^{[n]}) = p_{1*} \left[  \ch \left( \frac {\CO_{\Hilb_n \times S}}{\CI} \right) \cdot p_2^*(\ch(F) \td(S)) \right]
%\end{equation}
%The right-hand side is a particular case of a more general construction: \\

%$$
%c_1(F^{[n]}) = r \cdot p_{1*} \left[  \ch_3 \left( \frac {\CO_{\Hilb_n \times S}}{\CI} \right) \right] + p_{1*} \left[   \ch_3 \left( \frac {\CO_{\Hilb_n \times S}}{\CI} \right) \cdot p_2^* \left(c_1(F) + \td_1(S) \right) \right]
%$$

%with which the first Chern class of tautological bundles takes the form:
%\begin{equation}
%\label{eqn:c1 tautological}
%c_1(F^{[n]}) = r E + D_{c_1(F) - \frac t2}
%\end{equation}
%where $t = c_1(\CK_S) = - 2\td_1(S)$. \\

\begin{definition}
\label{def:small tautological}

Let $A^*_{\emph{small}}(\eHilb) \subset \eah$ denote the ring of \textbf{small} tautological classes, i.e. arbitrary sums of products of classes of the form:
\begin{equation}
\label{eqn:small tautological}
\pi_{1*} \Big[ \ech_k ( \CO_{\CZ} ) \cdot \pi_2^*(\gamma) \Big]
\end{equation}
where $k$ ranges over $\BN$ and $\gamma$ ranges over $R(S)$. \\

\end{definition}

\begin{definition}
\label{def:a big}
	
Let $A^*_{\emph{big}}(\eHilb) \subset \eah$ denote the ring of \textbf{big} tautological classes, i.e. arbitrary sums of products of classes of the form:
\begin{equation}
\label{eqn:big classes}
\pi_{1*} \Big[ \ech_{k_1}(\CO_{\CZ})... \ech_{k_t}(\CO_{\CZ}) \cdot \pi_2^*(\gamma) \Big]
\end{equation}
where $t,k_1,...,k_t$ range over $\BN$ and $\gamma$ ranges over $R(S)$. \\
	
\end{definition}

\noindent Tautological classes are closely related to tautological bundles, which are defined for every $n \in \BN$ and any rank $r$ vector bundle $V$ on $S$ by the construction:
$$
V^{[n]} = R\pi_{1*} \left( \CO_{\CZ_n} \otimes \pi_2^*(V) \right)
$$
Note that $V^{[n]}$ is a rank $rn$ vector bundle on $\Hilb_n$, and the Grothendieck- Hirzebruch-Riemann-Roch theorem implies that its Chern character is given by:
$$
\ch(V^{[n]}) = \pi_{1*} \Big( \ch(\CO_{\CZ_n}) \cdot \pi_2^*(\ch(V)) \cdot \pi_2^*(\text{td}(S))  \Big)
$$
If the Chern character of $V$ lies in the Beauville-Voisin subring $R(S)$, then the formula above shows that $\ch(V^{[n]})$ is a small tautological class (because $\text{td}(S) = 1 + 2c$, see \eqref{eqn:todd}). In particular, the first Chern class of $V^{[n]}$ is given by:
\begin{equation}
\label{eqn:c1 taut}
c_1(V^{[n]}) = \pi_{1*}([\CZ_n] \cdot c_1(V)) + \pi_{1*}(\ch_3(\CO_{\CZ_n}) \cdot r)
\end{equation}
The Picard groups of Hilbert schemes of points were described by Fogarty, whose Theorem 6.2 of \cite{Fo} shows that when $S$ is a K3 surface, $A^1(\Hilb_n)$ is generated by \eqref{eqn:c1 taut} as $V$ goes over all the line bundles on $S$. Hence we conclude the following: \\

\begin{proposition}
\label{prop:divisor}

Any divisor class on $\eHilb_n$ is a small tautological class. \\

\end{proposition}

\subsection{}
\label{sub:k3} 

Since $S$ is a K3 surface, $\Hilb_n = \Hilb_n(S)$ is holomorphic symplectic (\cite{B}). Let us review this fact, by recalling the explicit construction of the non-degenerate pairing on the tangent bundle of $\Hilb_n$. For simplicity, we will work at the level of an arbitrary closed point $I \in \Hilb_n$, in which case it is known that:
$$
\Tan_I \Hilb_n = \Hom(I,\CO_S/I)
$$
The long exact sequence associated to $0 \rightarrow I \rightarrow \CO_S \rightarrow \CO_S/I \rightarrow 0$ induces:
\begin{multline}
0 \rightarrow \Hom(\CO_S/I,\CO_S/I) \xrightarrow{\cong} \Hom(\CO_S,\CO_S/I) \rightarrow \\ 
\rightarrow \Hom(I,\CO_S/I) \rightarrow \Ext^1(\CO_S/I,\CO_S/I) \label{eqn:long}
\end{multline}
It is easy the observe that the second horizontal arrow is an isomorphism, since:
\begin{equation}
\label{eqn:ext 0}
\Hom(\CO_S/I,\CO_S/I) \cong \CO_S/I
\end{equation}
Note that $\dim_{\BC} \CO_S/I = n$. Moreover, Serre duality and $\CK_S \cong \CO_S$ imply that:
\begin{equation}
\label{eqn:ext 2}
\Ext^2(\CO_S/I,\CO_S/I) \cong (\CO_S/I)^\vee
\end{equation}
is also an $n$--dimensional vector space. Since $\sum_{i=0}^2 (-1)^i \dim_{\BC} \Ext^i(\CO_S/I,\CO_S/I) = 0$ (the quantity $\sum_i (-1)^i \dim \Ext^i(F,G)$ is additive in both arguments, and it is easy to observe that it vanishes on skyscraper sheaves), we conclude that:
\begin{equation}
\label{eqn:ext 1}
\dim_{\BC} \Ext^1(\CO_S/I,\CO_S/I)  = 2n
\end{equation}
Since $\Hilb_n$ is smooth of dimension $2n$, the long exact sequence \eqref{eqn:long} implies that:
\begin{equation}
\label{eqn:tangent}
\Tan_I \Hilb_n \cong \Ext^1(\CO_S/I,\CO_S/I)
\end{equation}
(the isomorphism above is simply the Kodaira-Spencer map, if one regards the Hilbert scheme as the moduli space parametrizing the finite length sheaves $\CO_S/I$). Moreover, Serre duality implies that the vector space $\Ext^1(\CO_S/I,\CO_S/I)$ is self-dual, which proves that $\Hilb_n$ is holomorphic symplectic. \\

\begin{proposition}
\label{prop:chern tangent}

The Chern character of the tangent bundle of $\eHilb_n$ is:
\begin{equation}
\label{eqn:chern tangent}
\ech(\eTan \ \eHilb_n) = \pi_{1*} \Big[ \Big(\ech (\CO_{\CZ_n}) + \ech (\CO_{\CZ_n})^\vee - \ech (\CO_{\CZ_n})  \ech (\CO_{\CZ_n})^\vee  \Big) \pi_2^*(1+2c) \Big] 
\end{equation}
where $\pi_1,\pi_2 : \eHilb_n \times S \rightarrow \eHilb_n, S$ are the standard projections. \\

\end{proposition}

\begin{proof} If we combine \eqref{eqn:ext 0}, \eqref{eqn:ext 2} and \eqref{eqn:ext 1}, we conclude the following equality in the Grothendieck group of locally free sheaves on $\Hilb_n$:
$$
[\RHom(\CO_S/I, \CO_S/I)] = [\CO_S/I] + [\CO_S/I]^\vee - [\Tan_I \Hilb_n]
$$
The version of this equality as $I$ varies over the Hilbert scheme yields:
$$
[\Tan \ \Hilb_n] = R\pi_{1*} \left(\left[ \frac {\CO_{\Hilb_n \times S}}{\CI} \right] + \left[ \frac {\CO_{\Hilb_n \times S}}{\CI} \right]^\vee - \left[\frac {\CO_{\Hilb_n \times S}}{\CI} \right] \cdot \left[ \frac {\CO_{\Hilb_n \times S}}{\CI} \right]^\vee \right)
$$
The Grothendieck-Hirzebruch-Riemann-Roch theorem applied to the formula above yields \eqref{eqn:chern tangent}, as soon as one recalls that the Todd genus of a K3 surface is:
\begin{equation}
\label{eqn:todd}
\td(S) = 1 + \frac {c_1(S)}2 + \frac {c_1(S)^2+c_2(S)}{12} = 1 + 2c
\end{equation}
(the fact that $c_2(S) = 24c$ is precisely \eqref{eqn:bv 1}).

\end{proof}

\section{Representation theory of Hilbert schemes}
\label{sec:rep theory}

\subsection{}
\label{sub:nakajima}

Let us recall the Heisenberg algebra action introduced independently by Grojnowski (\cite{G}) and Nakajima (\cite{Nak}) on the Chow groups of Hilbert schemes on an arbitrary smooth projective surface $S$. We will mostly follow the presentation of Nakajima in the current subsection. For any $n \in \BN$, consider the closed subscheme:
$$
\Hilb_{d,d+n} = \Big\{(I \supset I') \text{ s.t. } I/I' \text{ is supported at a single }x \in S \Big\} \subset \Hilb_d \times \Hilb_{d+n}
$$
endowed with projection maps:
\begin{equation}
\label{eqn:diagram zk}
\xymatrix{& \Hilb_{d,d+n} \ar[ld]_{p_-} \ar[d]^{p_S} \ar[rd]^{p_+} & \\ \Hilb_{d} & S & \Hilb_{d+n}}
\end{equation}
that keep track of $I,x$ and $I'$, respectively. It was shown in \cite{Nak} that the locus above has dimension $2d+n+1$, and so Nakajima used it to define the correspondences:
\begin{equation}
\label{eqn:nakajima}
A^*(\Hilb) \xrightarrow{\fq_{\pm n}} A^*(\Hilb \times S)
\end{equation}
where $A^*(\Hilb) = \bigoplus_{d=0}^\infty A^*(\Hilb_d)$, given by:
\begin{align}
&\fq_n = (p_+ \times p_S)_* \circ p_-^* \label{eqn:nak exp 1} \\
&\fq_{-n} = (-1)^n (p_- \times p_S)_* \circ p_+^* \label{eqn:nak exp 2}
\end{align}
We also set $\fq_0 = 0$. Loosely speaking, one may think of the operators $\fq_n$ as a family of endomorphisms of $A^*(\Hilb)$ indexed by $A^*(S)$, so we write for any $\gamma \in A^*(S)$:
\begin{equation}
\label{eqn:nakajima class}
\fq_n(\gamma) = \pi_{1*} (\fq_n \cdot \pi_2^*(\gamma))
\end{equation}
as an operator $A^*(\Hilb) \rightarrow A^*(\Hilb)$, where $\pi_1,\pi_2 : \Hilb \times S \rightarrow \Hilb, S$ are the standard projections. The Heisenberg algebra action is encoded in the fact that the operators $\fq_n$ satisfy the following commutation relations (see Theorem 8.13 and Remark 8.15 (2) of \cite{Nak book} for reference):
\begin{equation}
\label{eqn:heisenberg}
[\fq_n, \fq_{n'}] = n \delta_{n+n'}^0 \cdot \text{Id}_\Hilb \times [\Delta]
\end{equation}
where both sides of the equation are $\BQ$--linear maps $A^*(\Hilb) \rightarrow A^*(\Hilb \times S \times S)$. In terms of the endomorphisms \eqref{eqn:nakajima class}, the commutation relation \eqref{eqn:heisenberg} reads:
\begin{equation}
\label{eqn:heisenberg alt}
[\fq_n(\gamma), \fq_{n'}(\gamma')] = n \delta_{n+n'}^0 \langle \gamma, \gamma' \rangle \cdot \text{Id}_\Hilb
\end{equation}
where $\langle \cdot, \cdot \rangle$ is the intersection pairing on $S$. 

\subsection{} 

We may generalize the notation above to products of Nakajima operators:
\begin{equation}
\label{eqn:composition 1}
\fq_{n_1}...\fq_{n_k} : A^*(\Hilb) \longrightarrow A^*(\Hilb \times S^k)
\end{equation}
where the convention is that the operator $\fq_{n_i}$ acts in the $i$-th factor of $S^k = S \times ... \times S$. There are two related operations that we will apply in conjunction with products as \eqref{eqn:composition 1}. The first one is to restrict the composition to the smallest diagonal:
\begin{equation}
\label{eqn:composition 2}
\fq_{n_1}...\fq_{n_k} \Big|_\Delta : A^*(\Hilb) \xrightarrow{\fq_{n_1}...\fq_{n_k}} A^*(\Hilb \times S^k) \xrightarrow{\text{Id}_\Hilb \boxtimes \Delta^*} A^*(\Hilb \times S)
\end{equation}
and the second one is to use any $\Gamma \in A^*(S^k)$ to yield endomorphisms of $A^*(\Hilb)$:
\begin{equation}
\label{eqn:composition 3}
\fq_{n_1}...\fq_{n_k}(\Gamma) = \pi_{1*} (\fq_{n_1}...\fq_{n_k} \cdot \pi_2^*(\Gamma))
\end{equation}
where $\pi_1, \pi_2 : \Hilb \times S^k \rightarrow \Hilb, S^k$ denote the standard projections. The two operations \eqref{eqn:composition 2} and \eqref{eqn:composition 3} are related by the formula:
$$
\fq_{n_1}...\fq_{n_k} \Big|_\Delta(\gamma) = \fq_{n_1}...\fq_{n_k}(\Delta_*(\gamma))
$$
for any $\gamma \in A^*(S)$, where $\Delta$ refers to the smallest diagonal embedding $S \hookrightarrow S^k$. \\

\subsection{} The $\BQ$-vector space $A^*(\Hilb) = \oplus_{d=0}^\infty A^*(\Hilb_d)$ is graded by $d$, and it is clear from \eqref{eqn:diagram zk} that the operators $\fq_n$ increase this grading by $n$. When writing a product of the form \eqref{eqn:composition 1}, one may always use \eqref{eqn:heisenberg} to reorder all the terms, in such a way that $n_1 \geq ... \geq n_k$. More concretely, let us consider the \emph{normal-ordered product}:
\begin{equation}
\label{eqn:normal}
:\fq_a \fq_b: = \begin{cases} \fq_a \fq_b &\text{if } a \geq b \\ \fq_b \fq_a &\text{if } a < b \end{cases}
\end{equation}
The obvious generalization defines the normal-ordered products of several Heisenberg operators $:\fq_{n_1}...\fq_{n_k}:$. Note that the normal-ordered product only differs from the usual product if an operator $\fq_n$ with $n<0$ (called an \emph{annihilation operator}) is to the left of an operator $\fq_n$ with $n>0$ (called a \emph{creation operator}). Therefore, the normal-ordering convention can be explained, in words, as saying that all creation operators should be placed to the left of all annihilation operators. \\

\noindent It is easy to see that infinite expressions such as $\sum_{a+b = k}^{a,b \in \BZ} \fq_a \fq_b$ are not well defined on $\ah$. However, they do become well-defined if we normal-order all the products, as in the following analogues of Lehn's operators from cohomology:
\begin{equation}
\label{eqn:vir def}
\fL_n = \frac 12 \sum_{a+b = n}^{a,b \in \BZ} : \fq_a \fq_b: \Big|_\Delta 
\end{equation}
The reason why the $\BQ$--linear map $\fL_n : \ah \rightarrow \ahs$ is well-defined is that all the annihilation operators are to the right of all creation operators, and therefore $\fL_n$ acts by a finite sum on any given vector in $\ah$. The following formulas are straightforward consequences of \eqref{eqn:heisenberg} and \eqref{eqn:heisenberg alt}, and they are part of the fundamental motivation for Lehn's introduction of the operators \eqref{eqn:vir def}:
\begin{equation}
\label{eqn:heis vir}
[\fL_n, \fq_{n'}] = - n' \Delta_*(\fq_{n+n'}) \quad \Rightarrow \quad [\fL_n(1), \fq_{n'}(\gamma)] = -n'\fq_{n+n'}(\gamma) 
\end{equation}
\begin{multline}
[\fL_n, \fL_{n'}] = (n-n')\Delta_*(\fL_{n+n'}) - \frac {n^3-n}{12} \delta_{n+n'}^0 \cdot \text{Id}_\Hilb \boxtimes [\Delta_*(e)] \label{eqn:vir vir} \\
\Rightarrow \quad [\fL_n(1), \fL_{n'}(1)] = (n-n')\fL_{n+n'}(1) - \frac {n^3-n}{12} \delta_{n+n'}^0 \cdot e
\end{multline}
where $e = c_2(\Tan_S)$. \\

%\begin{proposition}
%\label{prop:vir 1}
	
%	For any $m,n,k \in \BZ$ and $\gamma \in \as$, we have the relations:
%	\begin{multline}
%	[\fq_m \fq_n |_\Delta, \fq_k] = \Delta_*(m \delta^{-k}_m \fq_n + n \delta^{-k}_n \fq_m)  \\ \Rightarrow \quad [\fq_m\fq_n(\Delta), \fq_k(\gamma)] = m \delta^{-k}_m \fq_n(\gamma) + n \delta^{-k}_n \fq_m(\gamma)
%	\end{multline}
%	\begin{multline}
%	[\fq_m \fq_n |_\Delta, \fq_{m'}\fq_{n'}|_\Delta] = \Delta_*(m \delta_m^{-m'} \fq_n\fq_{n'}|_\Delta + n \delta_n^{-m'} \fq_m\fq_{n'}|_\Delta + m \delta_m^{-n'} \fq_{m'}\fq_n|_\Delta + n \delta_n^{-n'} \fq_{m'}\fq_{m}|_\Delta) \\  \Rightarrow  [\fq_m \fq_n (\Delta), \fq_{m'}\fq_{n'} (\Delta)] = m \delta_m^{-m'} \fq_n\fq_{n'} (\Delta) + n \delta_n^{-m'} \fq_m\fq_{n'} (\Delta) + m \delta_m^{-n'} \fq_{m'}\fq_n (\Delta) + n \delta_n^{-n'} \fq_{m'}\fq_{m} (\Delta)
%	\end{multline}
%In both displays above, the expression to the left of the $\Rightarrow$ sign is a map $\eah \rightarrow \eahss$, while the expression to the right is a map $\eah \rightarrow \eah$. 
	
%\end{proposition}

\subsection{} Let us now consider the operators of multiplication by the Chern classes of the universal subscheme $\CZ \hookrightarrow \Hilb \times S$:
$$
\fG_k : \ah \xrightarrow{\pi_1^*} \ahs \xrightarrow{\cdot \ch_k(\CO_{\CZ})} \ahs 
$$
Because of \eqref{eqn:ch taut 1} and \eqref{eqn:ch taut 2}, we have $\fG_0 = \fG_1 = 0$. As before, we will write:
$$
\fG_k(\gamma) : \ah \xrightarrow{\fG_k} \ahs \xrightarrow{\cdot \pi_2^*(\gamma)} \ahs \xrightarrow{\pi_{1*}} \ah 
$$
for any $\gamma \in \as$. Alternatively, $\fG_k(\gamma)$ is the operator of multiplication by the small tautological class $\pi_{1*}(\ch_k(\CO_{\CZ})\cdot \pi_2^*(\gamma))$. One of the main goals of \cite{LQW} was to systematize the algebra generated by the operators $\fq_n$ and $\fG_k$, and the structure they found was that of the deformed $W_{1+\infty}$ algebra. Their construction was done in cohomology, but we will consider the exact same operators between Chow groups, and use them to prove Theorem \ref{thm:lqw}. Define:
\begin{equation}
\label{eqn:lqw}
\fJ_n^k : A^*(\Hilb) \rightarrow A^*(\Hilb \times S)
\end{equation}
\begin{equation}
\label{eqn:formula j}
\fJ_n^k = k! \left( - \sum^{|\lambda| = n}_{l(\lambda) = k+1} \frac {\fq_\lambda}{\lambda!}  \Big|_\Delta + \sum^{|\lambda| = n}_{l(\lambda) = k-1} (s(\lambda)+n^2-2) \frac {\pi_2^*(e) \cdot \fq_\lambda}{24\lambda!}  \Big|_\Delta \right) 
\end{equation}
%\begin{align}
%&J_0^a = - \sum^{|\lambda| = 0}_{l(\lambda) = a+1} \frac {\alpha_\lambda}{\lambda!}  \Big|_\Delta + \sum^{|\lambda| = 0}_{l(\lambda) = a-1} (s(\lambda)-2) \frac {c\alpha_\lambda}{\lambda!}  \Big|_\Delta \label{eqn:formula j0} \\
%&J_{\pm 1}^a = a! \left( - \sum^{|\lambda| = \pm 1}_{l(\lambda) = a+1} \frac {\alpha_\lambda}{\lambda!}  \Big|_\Delta + \sum^{|\lambda| = \pm 1}_{l(\lambda) = a-1} (s(\lambda)-1) \frac {c \alpha_\lambda}{\lambda!}  \Big|_\Delta \right) \label{eqn:formula j1}
%\end{align}
(note that our $\fJ_n^k$ are equal to the $\fJ_{-n}^k$ of \cite{LQW}) where $\lambda$ goes over all partitions of $\BZ \backslash \{0\}$. Let us explain the notation in \eqref{eqn:formula j}. Any partition $\lambda$ can be described as:
$$
\lambda = (...,(-2)^{m_{-2}}, (-1)^{m_{-1}}, 1^{m_1}, 2^{m_2},...)
$$
for $...m_{-2},m_{-1},m_1,m_2,... \in \BN \sqcup 0$, and we write $\fq_\lambda = ... \fq_2^{m_2} \fq_1^{m_1} \fq_{-1}^{m_{-1}} \fq_{-2}^{m_{-2}}...$ and:
$$
l(\lambda) = \sum_{i \in \BZ \backslash \{0\}} m_i, \quad |\lambda| = \sum_{i \in \BZ \backslash \{0\}} i m_i, \quad s(\lambda) = \sum_{i \in \BZ \backslash \{0\}} i^2 m_i, \quad \lambda! = \prod_{i \in \BZ \backslash \{0\}} m_i!
$$
We will now prove that the operators \eqref{eqn:lqw} satisfy the properties of Theorem \ref{thm:lqw}. \\ 

%We first remark that, as a consequence of relations \eqref{eqn:lqw rel 1}--\eqref{eqn:lqw rel 5}, the $W_{1+\infty}$ algebra generated by the operators $\fJ_n^k$ is: \\

%\begin{itemize}
	
%\item graded with $\fJ_n^k$ in degree $n$, in the sense that all relations are homogeneous with respect to $n$ degree \\

%\item filtered with $\fJ_n^k$ in the $k$-th filtered piece $F_k$, such that $\{\text{constants}\} = F_{-1} \subset F_0 \subset F_1 \subset ...$ and:
%$$
%x \in F_k, y \in F_l \quad \Rightarrow \quad xy \in F_{k+l}, [x,y] \in F_{k+l-1}
%$$
	
%\end{itemize}

\begin{proof} \emph{of Theorem \ref{thm:lqw}:} The fact that the operators \eqref{eqn:formula j} satisfy relations \eqref{eqn:lqw rel 1}--\eqref{eqn:lqw rel 5} is proved exactly as in \loccitt, since the only input necessary for their computation is the commutation relation \eqref{eqn:nakajima} of Nakajima operators. In particular, we have the following special cases of \eqref{eqn:lqw rel 1}--\eqref{eqn:lqw rel 5}, for all $a,k \geq 0$:
\begin{align}
&[\fJ_{\pm 1}^k, \fJ_0^2] = \mp 2 \Delta_* \left( \fJ_1^{k+1} \right) \label{eqn:lqw 4} \\
&[\fJ_1^a, \fJ_{-1}^{k-a}] = - k \Delta_* \left( \fJ_0^{k-1} \right) + k(k-1)(k-2) \Delta_* \left( \frac {\pi_2^*(e)}{12} \cdot \fJ_0^{k-3} \right)\label{eqn:lqw 5}
\end{align}
(the sign discrepancy between the formulas above and those of \loccit stems from the fact that our $\fJ_n^k$ are equal to their $\fJ_{-n}^k$). Therefore, to prove Theorem \ref{thm:lqw}, we only need to check \eqref{eqn:lqw 1} and \eqref{eqn:lqw 3}. The former is immediate, so it remains to prove the latter. The Grothendieck-Hirzebruch-Riemann-Roch theorem, together with $\ch_0(\CO_{\CZ}) = \ch_1(\CO_{\CZ}) = 0$, implies that:
$$
\pi_{1*} \left( \frac {\fJ_0^2}2  \right) = \pi_{1*} \left[ \text{multiplication by } \ch_3 \left( \CO_{\CZ_d} \right)  \right] = \text{multiplication by }c_1 (\CO_S^{[d]}) =: \fd
$$
where $\pi_1 : \Hilb_d \times S \rightarrow \Hilb_d$ is the standard projection. Note that \eqref{eqn:formula j} gives:
\begin{align*}
&\fJ_1^0 = - \fq_1 = - (p_+ \times p_S)_* \circ p_-^* \\
&\fJ_{-1}^0 = - \fq_{-1} = (p_- \times p_S)_* \circ p_+^*
\end{align*}
respectively, with the notation of \eqref{eqn:diagram zk}. Let: 
$$
\fr : \bigoplus_{d=0}^\infty A^*(\Hilb_{d,d+1}) \rightarrow \bigoplus_{d=0}^\infty A^*(\Hilb_{d,d+1})
$$
denote the operator of multiplication by $c_1(\CL)$, where $\CL$ is the tautological line bundle on $\Hilb_{d,d+1}$ whose fiber over $(I \supset I')$ is $I/I'$. Then we claim that:
\begin{align}
&\fJ_1^k = - (p_+ \times p_S)_* \circ \fr^k \circ p_-^* \label{eqn:higher nakajima 1} \\
&\fJ_{-1}^k = (p_- \times p_S)_* \circ \fr^k \circ p_+^* \label{eqn:higher nakajima 2}
\end{align}
Indeed, formulas \eqref{eqn:higher nakajima 1} and \eqref{eqn:higher nakajima 2} follow by comparing the fact that:
$$
[\fJ_{\pm 1}^k, \fd] = \mp \fJ_{\pm 1}^{k + 1} 
$$
(which follows from \eqref{eqn:lqw rel 1}) to the geometrically straightforward fact that:
$$
\left[ (p_\pm \times p_S)_* \circ \fr^k \circ p_\mp^* , \fd \right] = \mp (p_\pm \times p_S)_* \circ \fr^{k+1} \circ p_\mp^*
$$
Let us prove formula \eqref{eqn:lqw 3} by induction on $k$. The base cases $k=1$ and $k=2$ are precisely \eqref{eqn:degree} and \eqref{eqn:lehn}, respectively. As for the induction step, it is enough to invoke the $a=2$ case of \eqref{eqn:lqw 5} and prove the following equality of operators $A^*(\Hilb_d) \rightarrow A^*(\Hilb_d \times S \times S)$, for all $a,b \geq 0$:
\begin{multline}
\Big[ (p_+ \times p_S)_* \circ \fr^a \circ p_-^* , (p_- \times p_S)_* \circ \fr^b \circ p_+^* \Big] = \\
= \text{multiplication by } (a+b)! \cdot \Delta_*(\ch_{a+b}  (\CO_{\CZ})) \label{eqn:want}
\end{multline}
The left-hand side of \eqref{eqn:want} is the difference of operators:
\begin{align}
&(p_+ \times p_{S_1} \times \text{Id}_{S_2})_* \circ \fr^a \circ (p_- \times \text{Id}_{S_2})^* \circ (p_- \times p_{S_2})_* \circ \fr^b \circ p_+^* \label{eqn:comp 1} \\
&(p_- \times \text{Id}_{S_1} \times p_{S_2})_* \circ \fr^b \circ (p_+ \times \text{Id}_{S_1})^* \circ (p_+ \times p_{S_1})_* \circ \fr^a \circ p_-^* \label{eqn:comp 2}
\end{align}
(we write $S_1 = S_2 = S$, in order to differentiate between the two factors of the surface that appear in $\Hilb_d \times S_1 \times S_2$). As a cycle inside $\Hilb_d \times \Hilb_d \times S_1 \times S_2$, the composition \eqref{eqn:comp 1} (respectively \eqref{eqn:comp 2}) is supported on the locus $(I_1,I_2,x_1,x_2)$ such that there exists $I_1,I_2 \subset J$ (respectively $J' \subset I_1,I_2$) with $J/I_1 \cong \BC_{x_2}, J/I_2 \cong \BC_{x_1}$ (respectively $I_1/J' \cong \BC_{x_1}, I_2/J' \cong \BC_{x_2}$). On the open subset $(I_1,x_1) \neq (I_2,x_2)$, the aforementioned two loci are isomorphic via:
$$
(I_1,I_2 \subset J) \leadsto (J' \subset I_1,I_2) , \qquad J' = I_1 \cap I_2 \text{ inside }J
$$
This implies that the difference of \eqref{eqn:comp 1} and \eqref{eqn:comp 2}, i.e. the left-hand side of \eqref{eqn:want}, is a cycle supported on the diagonal $\Hilb_d \times S \hookrightarrow \Hilb_d \times \Hilb_d \times S \times S$. Hence:
\begin{equation}
\label{eqn:want 2}
\text{left-hand side of \eqref{eqn:want}} = \Big( \text{Id}_\Hilb \times \Delta \Big)_* \circ \Big( \text{multiplication by }\Gamma \Big)
\end{equation}
for some $\Gamma \in A^*(\Hilb_d \times S)$. Therefore, to prove \eqref{eqn:want} it suffices to show that:
\begin{equation}
\label{eqn:want 3}
\Gamma = (a+b)! \cdot \ch_{a+b}  \left( \CO_{\CZ_d} \right) 
\end{equation}
To prove that the class $\Gamma$ of \eqref{eqn:want 2} is given by \eqref{eqn:want 3}, it is enough to work out how the equality \eqref{eqn:want 2} of operators acts on the unit class $1 \in A^*(\Hilb_d)$. Explicitly, this boils down to the following computation, which will be proved in Subsection \ref{sub:nested scheme}: \\

\begin{claim}
\label{claim}

The following identity holds in $A^*(\eHilb_d \times S_1 \times S_2)$:
\begin{multline}
\label{eqn:want 4}
(p_+ \times p_{S_1} \times \emph{Id}_{S_2})_* \circ \fr^a \circ (p_- \times \emph{Id}_{S_2})^* \circ (p_- \times p_{S_2})_* (c_1(\CL)^b) - \\
- (p_- \times \emph{Id}_{S_1} \times p_{S_2})_* \circ \fr^b \circ (p_+ \times \emph{Id}_{S_1})^* \circ (p_+ \times p_{S_1})_* (c_1(\CL)^a) = \\ = (a+b)! \cdot \Delta_* \left( \ech_{a+b}  \left( \CO_{\CZ_d} \right) \right) 
\end{multline}

\end{claim}
	
\end{proof}

\section{The proof of the main Theorem}
\label{sec:proof}

\subsection{} 

Let us consider the following operators, in the notation \eqref{eqn:nakajima class} and \eqref{eqn:composition 3}:
\begin{align} 
&A^*(\Hilb) \xrightarrow{\fq_n(\gamma)} A^*(\Hilb) & &\forall \ \gamma \in R(S), \ n \in \BZ\backslash 0 \label{eqn:heis} \\
&A^*(\Hilb) \xrightarrow{L_n} A^*(\Hilb) & &L_n = \frac 12 \sum_{k+l = n} :\fq_k \fq_l:(\Delta^\trans) \label{eqn:vir}
\end{align}
where $\Delta^\tr$ denotes the class \eqref{eqn:transcendental}. Because of relations \eqref{eqn:bv 7} and \eqref{eqn:heis vir}, we have:
\begin{equation}
\label{eqn:commute}
[L_n, \fq_m(\gamma)] = 0 \qquad \forall \ n,m \in \BZ\backslash 0, \ \gamma \in R(S)
\end{equation}
and therefore the algebra generated by the operators \eqref{eqn:heis} and \eqref{eqn:vir} is:
$$
U(\heis \times \vir)
$$
where the Virasoro algebra with central charge $b$ is: 
$$
\vir = \BQ \Big \langle L_n \Big \rangle_{n \in \BZ} \Big / \left( [L_n,L_{n'} ] - (n-n')L_{n+n'} + \frac {n^3-n}{12} \delta_{n+n'}^0 b \right)
$$
(where $b \in \{2,...,21\}$ is the rank \eqref{eqn:b} of the transcendental lattice) and $\heis$ is the tensor product of $24-b = \dim_{\BQ} R(S)$ copies of the Heisenberg algebra. Explicitly, $\heis$ is generated by symbols $\fq_n(\gamma)$ as $n \in \BZ\backslash 0$ and $\gamma$ goes over a basis of $R(S)$, modulo relations \eqref{eqn:heisenberg alt}. \\

\begin{proposition}
\label{prop:in the algebra}
	
For any $k > 1$ and $\gamma \in R(S)$, the operators $\fG_k(\gamma)$ lie in the algebra $U(\emph{Heis} \times \emph{Vir})$. In virtue of Definition \ref{def:small tautological}, this implies that the operator of multiplication by any small tautological class lies in $U(\emph{Heis} \times \emph{Vir})$. \\
	
\end{proposition}

\begin{proof} Because of formula \eqref{eqn:transcendental}, we have:
$$
L_n = \fL_n(1) - \frac 12 \sum_{a+b = n} \left[ :\fq_a(c)\fq_b(1): + :\fq_a(1)\fq_b(c): + \sum_i :\fq_a(l_{(i)})\fq_b(l^{(i)}): \right]
$$
and it therefore suffices to show that the operators $\fG_k(\gamma)$ lie in the algebra generated by $\fq_n(\gamma)$ and $\fL_{n'}(1)$ (as $n$ goes over $\BZ\backslash 0$, $n'$ goes over $\BZ$ and $\gamma$ goes over $R(S)$). By comparing \eqref{eqn:lqw 3} with \eqref{eqn:formula j}, the operator $\fG_k$ is a sum of two terms, the first being:
\begin{equation}
\label{eqn:first}
- \frac 1{n!} \sum_{n_1+...+n_k = 0} :\fq_{n_1}...\fq_{n_k}:(\Delta_{12...k})
\end{equation}
and the second being: 
\begin{equation}
\label{eqn:second}
\sum_{n_1+...+n_{k-2} = 0} \text{coefficient} :\fq_{n_1}...\fq_{n_{k-2}}:(\Delta_{12...k-2}\cdot c_{k-2})
\end{equation}
for some coefficients in $\BQ$. We must show that both operators \eqref{eqn:first} and \eqref{eqn:second} lie in $U(\heis \times \vir)$. For the latter operator, this is clear, since \eqref{eqn:bv 3} allows us to write $\Delta_{12...k-2}\cdot c_{k-2} = c_1 \cdot c_2 \cdot ... \cdot c_{k-2}$, and so each summand in \eqref{eqn:second} is a product of operators $\fq_n(c) \in \heis$. As for \eqref{eqn:first}, the decomposition \eqref{eqn:bv 6} allows us to write:
$$
\text{equation \eqref{eqn:first}} = - \frac 1{n!} \sum_{n_1+...+n_k = 0} :\fq_{n_1}...\fq_{n_k}: \left( \sum_{1 \leq i < j \leq n} \Delta_{ij} \prod_{k \neq i,j} c_k - (n-2) \sum_{i=1}^n \prod_{k \neq i} c_k \right)
$$
$$
= - \frac 1{n!} \sum_{1 \leq i < j \leq n} \sum_{n_1+...+n_k = 0} :\fq_{n_1}(c_1)... \fq_{n_i}\fq_{n_j}(\Delta_{ij})...\fq_{n_k}(c_k): - ( \text{operator in }U(\heis) )
$$
For each fixed $i<j$ and each fixed $n_1,...,\widehat{n_i},...,\widehat{n_j},...,n_k$, the corresponding summand in the last line above is a product of $\fL_{-n_1-...-\widehat{n_i} - ... - \widehat{n_j} -...-n_k}(1)$ with the various $\fq_{n_a}(c)$, and so it lies in the algebra $U(\heis \times \vir)$. The reason for this fact is that, as we commute $q_{n_i}q_{n_j}(\Delta_{ij})$ with various other $q_{n_a}(c)$ in order to achieve the normally ordered product, the commutator lies in $U(\heis)$ by \eqref{eqn:heis vir}. 

\end{proof}

\subsection{} Let $V_{\sm} \subset A^*(\Hilb)$ denote the $\heis \times \vir$ module generated by: 
\begin{equation}
\label{eqn:vacuum}
v = 1 \in A^*(\Hilb_0) \cong \BQ
\end{equation}
Recall that $A^*(\Hilb) = \oplus_{n=0}^\infty A^*(\Hilb_n)$ is graded by $n$, and $L_0$ acts on the $n$-th graded subspace as multiplication by $n$. Aside from the word ``degree", we may refer to the eigenvalue of $L_0$ on a homogeneous element $w \in A^*(\Hilb)$ as the ``weight" of $w$. Because of this, the module $A^*(\Hilb)$ has lowest weight $0$. \\

\begin{corollary}
\label{cor:in the algebra} 

We have $A_{\emph{small}}^*(\eHilb) \subset V_{\emph{small}}$. \\

\end{corollary}

\begin{proof} The well-known formula $1_{\Hilb_n} = \frac 1{n!} \fq_1(1)^n(v)$ shows that the fundamental class of every Hilbert scheme $\Hilb_n$ lies in $V_\sm$. Therefore, Proposition \ref{prop:in the algebra} implies that the entire ring $A^*_\sm(\Hilb)$ lies inside $V_\sm$. 
	
%Let us first prove the inclusion $\subset$. We have the formula:
%$$
%1_{\Hilb_n} = \fq_1(1)^n(v)
%$$
%hence the fundamental class of every Hilbert scheme $\Hilb_n$ lies in $V_\sm$. Therefore, Proposition \ref{prop:in the algebra} implies that the entire ring $A^*_\sm(\Hilb)$ lies inside $V_\sm$. For the opposite inclusion $\supset$, it suffices to show that the operators in $\heis \times \vir$ preserve $A^*_\sm(\Hilb)$. This statement will be proved for $\fq_n(\gamma) \in \heis$ in Proposition \ref{prop:preserve}. As for the Virasoro operators, since these can be obtained by commuting $\fq_n(1)$ with $\fG_3(1)$ modulo $U(\heis)$ (see \eqref{eqn:lqw rel 1}), the fact that they preserve $A^*_\sm(\Hilb)$ follows.
	
\end{proof} 

\noindent The $\heis \times \vir$ module $V_\sm$ is of lowest weight, in the sense that is generated by:
\begin{equation}
\label{eqn:positive}
\{\fq_n(\gamma), L_{n'}\}_{n,n'>0}^{\gamma \in R(S)} \subset \heis \times \vir
\end{equation}
acting on the vector $v = 1 \in A^*(\Hilb_0) \cong \BQ$. We have:
\begin{equation}
\label{eqn:only relation}
\fq_n(\gamma) v = L_{n'}v = 0 \qquad \forall n < 0 , \ n' \leq 1
\end{equation}
for degree reasons (see the definition of Virasoro operators in \eqref{eqn:vir def}). \\

\begin{proposition} 
\label{prop:rep theory}
	
$V_{\emph{small}}$ is an irreducible $\emph{Heis} \times \emph{Vir}$ module. \\

\end{proposition} 

\begin{proof} Let $M \subset V_\sm$ denote a maximal proper submodule with respect to the $\text{Heis} \times \text{Vir}$ action. Fix a $\BQ$-basis $\Gamma$ of $R(S)$, and let us consider any element:
\begin{equation}
\label{eqn:sum}
\sum_{ \{(n_1,\gamma_1),...,(n_k,\gamma_k) \} \subset \BN \times \Gamma} \fq_{n_1}(\gamma_1)...  \fq_{n_k}(\gamma_k) v^{\gamma_1,...,\gamma_k}_{n_1,...,n_k} \in M
\end{equation}
where $v^{\gamma_1,...,\gamma_k}_{n_1,...,n_k} \in \vir \cdot v$. Since the integral pairing on $R(S)$ is non-degenerate, by applying the operators $\fq_{-n}(\gamma)$ for various $n>0, \gamma \in R(S)$ to the sum in \eqref{eqn:sum}, we infer that $v^{\gamma_1,...,\gamma_n}_{n_1,...,n_k} \in M$ for any unordered set $\{(n_1,\gamma_1),...,(n_k,\gamma_k) \} \subset \BN \times \Gamma$. Therefore, if we let $N \subset \vir \cdot v$ denote the maximal Virasoro algebra submodule spanned by the $v^{\gamma_1,...,\gamma_k}_{n_1,...,n_k}$ that appear in \eqref{eqn:sum} for various elements of $M$, we have:
$$
M \subset U(\heis) \cdot N
$$
However, the classification of lowest weight modules of $\vir$ from \cite{FF} shows that the unique such maximal proper submodule $N$ is necessarily generated by $L_1 v$ (since the central charge of our $\vir$, namely the rank $b$ of the transcendental lattice, is contained between 2 and 21, and the weight of $v$ is 0). Since $L_1 v = 0$ due to \eqref{eqn:only relation}, we conclude that $N=0$, hence $M = 0$ and thus $V_\sm$ is irreducible.

\end{proof}

\begin{proof} \emph{of Theorem \ref{thm:main}:} There exists a $\heis \times \vir$ action on $H^*(\Hilb)$ with respect to which the cycle class map $\zeta : A^*(\Hilb) \rightarrow H^*(\Hilb)$ is equivariant (the construction and proof of all statements in cohomology are analogous to those in Chow). Therefore, Proposition \ref{prop:rep theory} and Schur's Lemma imply that:
\begin{equation}
\label{eqn:zeta sm}
\zeta|_{V_\sm} : V_\sm \rightarrow H^*(\Hilb)
\end{equation}
is either 0 or injective. Since $\zeta$ is an isomorphism between the one-dimensional vector spaces $A^*(\Hilb_0)$ and $H^*(\Hilb_0)$, we conclude that \eqref{eqn:zeta sm} is injective. Together with Corollary \ref{cor:in the algebra}, this concludes the proof of Theorem \ref{thm:main}.
	
\end{proof}

\section{The representation theory of tautological classes}
\label{sec:big}

\subsection{} Let $V_\big \subset A^*(\Hilb)$ denote the $\heis$ submodule generated by:
\begin{equation}
\label{eqn:products}
\prod_{i=1}^t \fq_{m_i} \fq_{n_i}(\Delta) \cdot v
\end{equation}
over all $(m_1,n_1),...,(m_t,n_t) \in \BN^2$, where $v = 1 \in A^*(\Hilb_0) \cong \BQ$. \\

\begin{proposition} 
\label{prop:prop 1}

$V_{\emph{big}}$ is preserved by the operators $\fq_m \fq_n(\Delta)$, for all $m,n \in \BZ\backslash 0$. \\

\end{proposition} 

\noindent In particular, the Proposition above implies that $V_\big$ is also preserved by $\vir$, due to formula \eqref{eqn:vir def}. Therefore, we have $V_\big \supset V_\sm$. 

\begin{proof} It suffices to show that:
$$
\fq_m \fq_n (\Delta) \cdot \prod_i \fq_{k_i}(\gamma_i)  \prod_j \fq_{m_j} \fq_{n_j}(\Delta) \cdot v
$$
lies in $V_\big$ for any choice of indices. This follows from the commutation relations:
\begin{equation}
\label{eqn:oi}
[\fq_m \fq_n(\Delta), \fq_k(\gamma)] = m \delta_{m+k}^0 \fq_n(\gamma) + n \delta_{n+k}^0 \fq_m(\gamma) 
\end{equation}
and:
\begin{multline} \label{eqn:oioi}
[\fq_m \fq_n(\Delta), \fq_{m'} \fq_{n'}(\Delta)] = m \delta_{m+m'}^0 \fq_n \fq_{n'}(\Delta) + \\
+ m \delta_{m+n'}^0 \fq_{m'} \fq_{n}(\Delta) + n \delta_{n+m'}^0 \fq_m \fq_{n'}(\Delta) + n \delta_{n+n'}^0 \fq_{m'} \fq_{m}(\Delta) 
\end{multline}
which are simple consequences of \eqref{eqn:heisenberg}. 

\end{proof} 

\begin{proposition}
\label{prop:prop 2}

Any vector subspace $V \subset A^*(\eHilb)$ which contains $v$ and is preserved by both $\emph{Heis} \times \emph{Vir}$ and multiplication with $\ech_2(\emph{Tan})$ must contain $V_{\emph{big}}$. \\

\end{proposition}

\begin{proof} Proposition \ref{prop:in the algebra} states the operators $\{\fG_k(\gamma)\}^{k \in \BN}_{\gamma \in R(S)}$ lie in $U(\heis \times \vir)$. Therefore, Proposition \ref{prop:chern tangent} (together with $\ch_0(\CO_{\CZ}) = \ch_1(\CO_{\CZ}) = 0$) implies that:
\begin{multline*}
\text{multiplication by }\ch_2(\Tan) = \\ = - \text{multiplication by } \pi_{1*}\Big[\ch_2(\CO_{\CZ})\ch_2(\CO_{\CZ}) (1+2c)\Big]  \quad \text{mod } U(\heis \times \vir)
\end{multline*}
where $\Hilb \times S \xrightarrow{\pi_1,\pi_2} \Hilb, S$ are the standard projections. Meanwhile, the operator of multiplication on the second line is $\fL_0\fL_0|_\Delta(1+2\pi_2^*(c))$, by \eqref{eqn:degree}. Moreover:
$$
\fL_0\fL_0|_\Delta(2c) = 2 \fL_0\fL_0(\Delta \cdot c) = 2\fL_0\fL_0(c \cdot c) = 2\fL_0(c)^2 \in U(\heis \times \vir)
$$
(the second equality follows from \eqref{eqn:bv 3}), implies that we have:
$$
\text{multiplication by } \ch_2(\Tan) = - \fL_0\fL_0(\Delta) \text{ mod } U(\heis \times \vir)
$$
Hence if $V$ is to be preserved by $\heis \times \vir$ and multiplication by $\ch_2(\Tan)$, then it must also be preserved by $-\fL_0 \fL_0(\Delta)$. However, \eqref{eqn:heis vir} implies the relations:
$$
[-\fL_0 \fL_0(\Delta), \fq_m(1)] = m \fL_0\fq_m(\Delta) + m \fq_m \fL_0(\Delta) = 2m \fq_m\fL_0(\Delta) - m^2 \fq_m(24c) 
$$
which implies that $V$ must also be preserved by $\fq_m \fL_0(\Delta)$. Similarly, the relation:
$$
[\fq_m \fL_0(\Delta),\fq_n(1)] = - n \fq_m\fq_n(\Delta) + m \delta_{m+n}^0 \fL_0(1)
$$
implies that $V$ must also be preserved by $\fq_m\fq_n(\Delta)$. Since $V \ni v$, this implies that $V$ must also contain all the vectors \eqref{eqn:products}, and therefore $V \supset V_\big$. 	
	
\end{proof}

\noindent The Proposition above shows that any generalization of the proof of Theorem \ref{thm:main} that accounts for the operators of multiplication by Chern classes of the tangent bundle must necessarily contend with the vector space $V_\big$. The following Proposition shows that this vector space in fact contains all big tautological classes. \\

\begin{proposition}
\label{prop:prop 3}

We have $A_{\emph{big}}^*(\eHilb) = V_{\emph{big}}$. \\

\end{proposition}

\begin{proof} Let us first prove the inclusion $\subset$. By definition, the ring $A_\big(\Hilb)$ is generated by the classes \eqref{eqn:big classes}. The operator of multiplication by \eqref{eqn:big classes} is:
$$
\fG_{k_1}...\fG_{k_t}(\Delta_{12...t}\cdot \gamma_t) : \ah \rightarrow \ah
$$
By Theorem \ref{thm:lqw} (specifically formula \eqref{eqn:formula j}), the operator above can be written as a linear combination of operators of the form:
\begin{equation}
\label{eqn:above}
\fq_{n_1}...\fq_{n_s}(\Delta_{12...s} \cdot \gamma'_s)
\end{equation}
for some $\gamma' \in R(S)$. By applying \eqref{eqn:bv 6}, one can write expression \eqref{eqn:above} as a product of operators in $\heis$ with a single operator of the form $\fq_{n_i} \fq_{n_j}(\Delta)$. As both kinds of operators preserve $V_\big$ (the former by definition, the latter by Proposition \ref{prop:prop 1}), we conclude that $A^*_\big(\Hilb) \subset V_\big$. The inclusion $\supset$ follows from Propositions \ref{prop:preserve}. 

\end{proof} 

\begin{proof} \emph{of Proposition \ref{prop:equiv}:} The argument below closely follows the final remark of \cite{Y}. Recall the following result of de Cataldo and Migliorini (\cite{dCM}, Theorem 5.4.1):
\begin{equation}
\label{eqn:decomp}
A^*(\Hilb) = \bigoplus_{n_1 \geq ... \geq n_k > 0} \BQ \cdot \fq_{n_1}... \fq_{n_k}(\Gamma) \cdot v
\end{equation}
as $\Gamma$ runs over a $\BQ$--basis of $A(S^k)^{\text{sym}}$, where $\text{sym}$ denotes the part which is symmetric with respect to those transpositions $(ij) \in \Sigma_k$ for which $n_i = n_j$. \\

\noindent  Let us first show that Conjecture \ref{conj:main} implies Conjecture \ref{conj:v}. To this end, suppose $\Gamma \in R(S^k)$ is such that $\bar{\zeta}(\Gamma) = 0 \in H^*(S^k)$, where $\bar{\zeta} : A^*(S^k) \rightarrow H^*(S^k)$ denotes the cycle class map. Since the cycle class map commutes with the assignment:
$$
\Gamma \leadsto \fq_{n_1}...\fq_{n_k}(\Gamma)
$$
we conclude that $\fq_{n_1}...\fq_{n_k}(\Gamma) \cdot v = 0 \in H^*(\Hilb_{n_1+...+n_k})$ for any $n_1,...,n_k \in \BN$. By the very definition of $R(S^k)$ and relations \eqref{eqn:bv 1}--\eqref{eqn:bv 7}, the class $\Gamma$ can be written as a product of pairwise diagonals $(p_i \times p_j)^*(\Delta)$ and classes $p_i^*(l)$, $p_i^*(c)$, for various $1\leq i < j \leq k$, where $p_i : S^k \rightarrow S$ denotes the $i$-th projection map. Therefore: 
$$
\fq_{n_1}...\fq_{n_k}(\Gamma) \cdot v \in V_\big \stackrel{\text{Prop. \ref{prop:prop 3}}}= A^*_\big(\Hilb)
$$
Conjecture \ref{conj:main} then implies that $\fq_{n_1}...\fq_{n_k}(\Gamma) \cdot v = 0 \in A^*(\Hilb_{n_1+...+n_k})$, and if the numbers $n_1,...,n_k$ are taken to be distinct, then \eqref{eqn:decomp} implies that $\Gamma = 0 \in A^*(S^k)$. \\

\noindent Conversely, let us show that Conjecture \ref{conj:v} implies Conjecture \ref{conj:main}. By Propositions \ref{prop:prop 1} and  \ref{prop:prop 3}, it suffices to show that the cycle class map $\zeta : A^*(\Hilb) \rightarrow H^{*}(\Hilb)$ is injective on the $\BQ$-span of:
\begin{equation}
\label{eqn:tbi}
\prod_{i=1}^k \fq_{m_i} \fq_{n_i} (\Delta) \prod_{j=1}^l \fq_{p_j}(\gamma^{(j)}) \cdot v = \fq_{m_1}\fq_{n_1}...\fq_{m_k} \fq_{n_k} \fq_{p_1}...\fq_{p_l}(\Gamma) \cdot v \in A^*(\Hilb)
\end{equation}
for any natural numbers $m_i,n_i,p_i$ and any classes $\gamma^{(j)} \in R(S)$, where we write:
\begin{equation}
\label{eqn:bat}
\Gamma = \Delta_{12} \Delta_{34}...\Delta_{2k-1,2k} \prod_{j=1}^l \gamma^{(j)}_{2k+j}
\end{equation}
Recall that \eqref{eqn:decomp} states that any linear relation between $\zeta(\text{elements \eqref{eqn:tbi}})$ implies the corresponding linear relation between $\zeta(\text{Sym}(\text{elements \eqref{eqn:bat}}))$ (here \text{Sym} denotes the operator of symmetrization with respect to the subgroup of permutations generated by transpositions corresponding to those pairs of numbers $m_i, n_i, p_j$ which are equal). Since the latter elements actually lie in $\zeta(R(S^{2k+l}))$, Conjecture \ref{conj:v} yields a linear relation between the Sym applied to the elements \eqref{eqn:bat} in the Chow group of $S^{2k+l}$. Plugging this relation back into $\fq_{m_1}\fq_{n_1}...\fq_{m_k} \fq_{n_k} \fq_{p_1}...\fq_{p_l}(...)$ implies a linear relation between the elements \eqref{eqn:tbi}, as required. 

\end{proof}

\subsection{} In the remainder of this Section, we will develop the representation theory of the space $V_\big$. We may consider the operators (in the notation of \eqref{eqn:transcendental}):
$$
\fq_m \fq_n(\Delta^\tr) = \fq_m\fq_n(\Delta) - \fq_m(c)\fq_n(1) - \fq_m(1)\fq_n(c) - \sum_i \fq_m(l_{(i)}) \fq_n(l^{(i)})
$$
Similar with \eqref{eqn:commute}, we have:
$$
[\fq_m\fq_n(\Delta^\tr), \fq_p(\gamma)] = 0 \qquad \forall m,n,p \in \BZ\backslash 0, \ \gamma \in R(S)
$$
and therefore the vector space $V_\big$ of \eqref{eqn:products} factors as:
$$
V_\big = \text{Fock} \otimes W
$$
where $\text{Fock} = \text{Heis} \cdot v$ and:
$$
W = \bigoplus^{\text{unordered collections}}_{(m_1,n_1),...,(m_t,n_t) \in \BN^2} \BQ \cdot \prod_{i=1}^t \fq_{m_i} \fq_{n_i}(\Delta^\tr) \cdot v
$$
By analogy with Proposition \ref{prop:prop 1}, the vector space $W$ is preserved by the operators $\fq_m \fq_n(\Delta^\tr)$ for all $m,n \in \BZ \backslash 0$. Therefore, we will study the algebra generated by these operators, or more precisely, their renormalized versions:
\begin{equation}
\label{eqn:sp}
D_{m,n} = \frac {\sgn n}{\sqrt{|mn|}} \fq_{m} \fq_{n}(\Delta^\tr) + \delta_{m+n}^0 \frac {b}2 \cdot \text{Id}_{A^*(\Hilb)}
\end{equation}
for all $m \geq n \in \BZ\backslash 0$, where $b$ is the rank of the transcendental lattice. As suggested by Pavel Etingof, the operators \eqref{eqn:sp} generate a well-known Lie algebra: \\

\begin{definition}
\label{def:sp}

Let us consider matrices with infinitely many rows and columns, both indexed by $\BZ \backslash 0$. Let $\fg = \fsp_{2\infty}$ be the Lie algebra of such matrices where all but finitely many entries are 0, the submatrices with $(\text{rows}, \text{columns})$ indexed by $\BN \times (-\BN)$ and $(-\BN) \times \BN$ are symmetric, and the submatrix with $(\text{rows}, \text{columns})$ indexed by $(-\BN) \times (-\BN)$ is the negative transpose of the submatrix indexed by $\BN \times \BN$. \\
 
\end{definition}

\noindent In more detail, $\fg = \fsp_{2\infty}$ is the direct limit of the finite-dimensional Lie algebras $\mathfrak{sp}_{2N}$ as $N \rightarrow \infty$. Let us consider the elements:
\begin{equation} 
\label{eqn:sp gen}
\fg \ni d_{m,n} = E_{m, -n} + (\sgn m)(\sgn n)E_{n,-m} 
\end{equation} 
for any $m, n \in \BZ\backslash 0$, where $E_{m,n}$ denotes the elementary matrix with a single 1 at the intersection of row $m$ and column $n$, and 0 elsewhere. It is elementary to observe that the elements \eqref{eqn:sp gen} with $m \geq n$ generate $\fg$, and that they satisfy the following commutation relations:
\begin{multline}
\label{eqn:sp rel}
[d_{m,n}, d_{m',n'}] = \delta_{n+m'}^0 d_{m,n'} + \delta_{m+m'}^0 (\sgn m)(\sgn n) d_{n,n'} + \\ + \delta_{n+n'}^0 (\sgn m')(\sgn n') d_{m,m'} + \delta_{m+n'}^0 (\sgn m)(\sgn n)(\sgn m')(\sgn n') d_{n,m'}
\end{multline}
for all $m,n,m',n'$. \\

\begin{proposition}

The operators \eqref{eqn:sp} give an action of $\fg = \fsp_{2\infty}$ on $A^*(\eHilb)$. \\

\end{proposition}

\noindent The Proposition follows by comparing \eqref{eqn:oioi} (with $\Delta^\tr$ instead of $\Delta$) with \eqref{eqn:sp rel}. The occurence of $b$ stems from the fact that \eqref{eqn:b} implies that:
$$
\fq_{-n} \fq_n(\Delta^\tr) = \fq_n \fq_{-n}(\Delta^\tr) - n b \cdot \text{Id}
$$

\subsection{} Let us now analyze the submodule $W \subset A^*(\Hilb)$ generated by the operators \eqref{eqn:sp} acting on the vacuum vector $v$. It is easy to see that:
\begin{equation}
\label{eqn:highest}
D_{m,n} \cdot v = \frac b2 \cdot \delta_{m+n}^0 \qquad \forall m \geq n, 0 > n
\end{equation}
Therefore, we conclude that there exists a surjective map:
\begin{equation}
\label{eqn:verma}
M := U(\fg) \otimes_{U(\fp)} \BQ_\chi \twoheadrightarrow W
\end{equation}
where $\fp \subset \fg$ is the parabolic subalgebra consisting of all matrices with zeroes in the $(-\BN) \times \BN$ block, and the character $\chi : \fp \rightarrow \BQ$ is given by:
$$
\chi\begin{pmatrix} 0 & A \\ -A^T & B \end{pmatrix} = \frac b2 \cdot \text{Tr}(A)
$$
The $\fg$ module $M$ is called a parabolic Verma module, and we will write $v_\emptyset$ for the element $1 \otimes 1 \in M$. If we let $\ft \subset \fg$ be the Cartan subalgebra of diagonal matrices, then the weights of the Lie algebra $\fg$ can be expressed as:
\begin{equation}
\label{eqn:weights}
a_1 \e_1 + ... + a_n \e_n + ...
\end{equation}
where $a_i \in \BQ$ and $\e_i$ is the dual basis to the matrices $E_{-i,-i} - E_{i,i} \in \ft$ (to make sense of the weights of $\fsp_{2\infty}$, one must present this Lie algebra as the limit of $\fsp_{2N}$ as $N \rightarrow \infty$, whose weights take the form \eqref{eqn:weights} with $n$ up to $N$). The highest weight of the parabolic Verma module $M$ is $-b/2(\e_1+...+\e_n+...)$. Let us consider: 
$$
L \subset M
$$
to be the $\fg$--submodule generated by expressions:
\begin{equation}
\label{eqn:yin rep}
v^{m_1,...,m_{b+1}}_{n_1,...,n_{b+1}} = \sum_{\sigma \in \Sigma_{b+1}} (\sgn \sigma) d_{m_1,n_{\sigma(1)}} ... d_{m_{b+1},n_{\sigma(b+1)}} \cdot v_\emptyset
\end{equation}
as $m_1 < ... < m_{b+1}$ and  $n_1 < ... < n_{b+1}$ go over $\BN$. Note that the vectors \eqref{eqn:yin rep} correspond to the left-hand side of the Kimura relation \eqref{eqn:yin}, under the de Cataldo-Migliorini correspondence between $S^{2b+2}$ and $\Hilb$. Then Conjectures \ref{conj:main} and \ref{conj:v} would follow from the fact that \eqref{eqn:verma} factors through a map of $\fg$--modules:
$$
M/L \twoheadrightarrow W
$$
since this would ensure that the Kimura relation \eqref{eqn:yin} holds in Chow. Recall that the Levi subgroup $\fh$ of $\fp$ corresponds to submatrices whose rows and columns are indexed by $\BN$, and as such $\fh \cong \fgl_\infty$. We have the tautological representation $\fh \curvearrowright \BC^\infty$ with basis vectors $e_1,e_2,...$, rescaled so that:
\begin{equation}
\label{eqn:kin}
\fgl_\infty \ni \left(E_{m,n} - \frac b2 \delta_m^n \right) \cdot e_p = \delta_n^p e_m
\end{equation}
and we consider the representation $R = S^2(\wedge^{b+1} \BC^{\infty})$. \\

\begin{proposition}

There is a map of $\fg$-modules $U(\fg) \otimes_{U(\fp)} R \twoheadrightarrow L$ induced by:
\begin{equation}
\label{eqn:assignment}
1 \otimes (e_{m_1} \wedge ... \wedge e_{m_{b+1}})(e_{n_1} \wedge ... \wedge e_{n_{b+1}}) \leadsto v^{m_1,...,m_{b+1}}_{n_1,...,n_{b+1}}
\end{equation}
for all natural numbers $m_1 < ... < m_{b+1}$ and $n_1 < ... < n_{b+1}$. \\

\end{proposition}

\begin{proof} First of all, let us prove that the nilpotent subalgebra of $\fp$ annihilates the right-hand side of \eqref{eqn:assignment}. To keep the notation simple, we will do it in the case when $m_i = n_i = i$, and leave the general case to the interested reader. We must prove the following for all $m,n > 0$ (the hats denote missing terms):
$$
d_{-m,-n} \cdot v_{1,...,b+1}^{1,...,b+1} = \sum_{\sigma \in \Sigma_{b+1}} (\sgn \sigma) d_{-m,-n} d_{1,\sigma(1)} ... d_{b+1,\sigma(b+1)} \cdot v_\emptyset \stackrel{\eqref{eqn:sp rel}} = \downarrow^{b+1}_{m-1}
$$
$$
\sum_{\sigma \in \Sigma_{b+1}} (\sgn \sigma) \Big[  d_{1,\sigma(1)}... d_{-n,\sigma(m)} ... d_{b+1,\sigma(b+1)} + d_{1,\sigma(1)}... d_{-n,\sigma^{-1}(m)} ... d_{b+1,\sigma(b+1)} \Big] v_\emptyset + (...)  
$$
where the summands marked by $(...)$ are the same ones as the terms directly preceding them, but with $m$ replaced by $n$. The symbol $\downarrow^{a}_{b}$ is 1 if $a>b$ and 0 otherwise. By \eqref{eqn:sp rel} and \eqref{eqn:highest}, the formula above equals:
\begin{equation}
\label{eqn:label}
\downarrow^{b+1}_{m-1} \downarrow^{b+1}_{n-1} \sum_{\sigma \in \Sigma_{b+1}} (\sgn \sigma) \left[ - b \delta_{\sigma(m)}^n d_{1,\sigma(1)}... \widehat{d_{m,n}}... d_{b+1, \sigma(b+1)}   + \downarrow^{\sigma^{-1}(n)}_{\sigma^{-1}(m)} \right.
\end{equation}
$$
... \widehat{d_{\sigma^{-1}(m),m}} ... \widehat{d_{\sigma^{-1}(n),n}} ... d_{\sigma^{-1}(m), \sigma^{-1}(n)} + \downarrow^{n}_{\sigma^{-1}(m)} ... \widehat{d_{\sigma^{-1}(m),m}} ... \widehat{d_{n,\sigma(n)}} ... d_{\sigma^{-1}(m),\sigma(n)} 
$$
$$
\left. +  \downarrow^{n}_{m} ... \widehat{d_{m,\sigma(m)}} ... \widehat{d_{n,\sigma(n)}} ... d_{\sigma(m),\sigma(n)} + \downarrow^{\sigma^{-1}(n)}_{m}... \widehat{d_{m,\sigma(m)}} ... \widehat{d_{\sigma^{-1}(n),n}} ... d_{\sigma(m),\sigma^{-1}(n)} \right] 
$$
$+(...)$. We claim that the expression above is 0. To see this, note that as $\sigma$ varies, the terms in the last two rows of the expression above (plus the corresponding two rows when $m$ and $n$ are switched, which are encoded in the summands denoted $(...)$) are in 1-to-1 correspondence to the outputs of the following algorithm: \\

\begin{itemize}
	
	\item draw a perfect matching between a set of red balls labeled by $1,...,b+1$ and a set of yellow balls labeled by $1,...,b+1$ (this corresponds to $\sigma$) \\
 	
	\item find any two balls labeled by $m$ and $n$, remove them, and then match together their former matches \\

\end{itemize}

\noindent If the two balls which were removed had the same color, their corresponding terms would cancel out from \eqref{eqn:label} due to the presence of the signature $\sgn \sigma$. If the two balls have different colors, then their corresponding terms are precisely canceled by the summand on the third line of \eqref{eqn:label}, which implies the fact that the total sum equals 0, as required. \\

\noindent The second thing we need to prove is that the action induced by the Levi subgroup $\fh \cong \fgl_\infty \subset \fg$ on the two sides of \eqref{eqn:assignment} is well-defined. To this end, let us identify the generator $d_{m,-n} \in \fh$ with the $\infty \times \infty$ matrix $E_{m,n}$ with entry 1 at the intersection of row $m$ and column $n$, and 0 everywhere else. As a consequence of \eqref{eqn:kin}:
$$
\left(E_{s,u} - \frac b2 \delta_s^u \right) \cdot (e_{m_1} \wedge ... \wedge e_{m_{b+1}})(e_{n_1} \wedge ... \wedge e_{n_{b+1}}) = 
$$
$$
= \sum^{i \text{ s.t.}}_{m_i = u} ( ... \wedge e_{m_{i-1}} \wedge e_s \wedge e_{m_{i+1}} \wedge ...)(e_{n_1} \wedge ... \wedge e_{n_{b+1}}) + 
$$
\begin{equation}
\label{eqn:comm 1}
+ \sum^{i \text{ s.t.}}_{n_i = u} (e_{m_1} \wedge ... \wedge e_{m_{b+1}})( ... \wedge e_{n_{i-1}} \wedge e_s \wedge e_{n_{i+1}} \wedge ...)
\end{equation}
is the sum of all terms obtained by all ways of isolating $e_u$ in the two wedge products, and replacing them by $e_s$. Similarly, formula \eqref{eqn:sp rel} implies that:
$$
\left(d_{s,-u} - \frac b2 \delta_s^u \right) v^{m_1,...,m_{b+1}}_{n_1,...,n_{b+1}} = \sum_{\sigma \in \Sigma_{b+1}} (\sgn \sigma) [d_{s,-u}, d_{m_1,n_{\sigma(1)}} ... d_{m_{b+1},n_{\sigma(b+1)}}] \cdot v_\emptyset =
$$
$$
= \sum_{\sigma \in \Sigma_{b+1}} (\sgn \sigma) \left[ \sum^{i \text{ s.t.}}_{m_i = u} ... d_{m_{i-1}, n_{\sigma(i-1)}} d_{s,n_{\sigma(i)}} d_{m_{i+1}, n_{\sigma(i+1)}} ...  +  \right. 
$$
$$
\left. + \sum^{i \text{ s.t.}}_{n_i = u} ... d_{m_{\sigma^{-1}(i-1)}, n_{i-1}} d_{s,m_i} d_{m_{\sigma^{-1}(i+1)}, n_{i+1}} ...  \right] \cdot v_\emptyset = 
$$
\begin{equation}
\label{eqn:comm 2}
= \sum^{i \text{ s.t.}}_{m_i = u} v_{...n_{i-1},n_i,n_{i+1},...}^{...,m_{i-1},s,m_{i+1},...} + \sum^{i \text{ s.t.}}_{n_i = u} v_{...n_{i-1},s,n_{i+1},...}^{...,m_{i-1},m_i,m_{i+1},...}
\end{equation}
Comparing \eqref{eqn:comm 1} with \eqref{eqn:comm 2} implies that \eqref{eqn:assignment} is a map of $\fgl_\infty$ modules. 
	
\end{proof}

\section{The geometry of nested Hilbert schemes}
\label{sec:lehn}

\subsection{}
\label{sub:nested}

The main purpose of the current Section is to prove Theorem \ref{thm:lehn}. Therefore, we let $S$ be an arbitrary smooth projective surface over $\BC$ for the remainder of this paper (in other words, we drop the K3 assumption), and let $\CI$ denote the universal ideal sheaf on $\Hilb_n \times S$. Because $\CI$ is flat over $\Hilb_n$, it inherits properties from the ideals of $\CO_S$ it parametrizes, such as having homological dimension 1: \\

\begin{proposition}
\label{prop:length 1}

(\cite{Shuf}) There exists a short exact sequence on $\eHilb_n \times S$:
\begin{equation}
\label{eqn:length 1}
0 \rightarrow \CW \rightarrow \CV \rightarrow \CI \rightarrow 0
\end{equation}
with $\CW$ and $\CV$ locally free. \\

\end{proposition}

%\begin{proof} Consider an ample divisor $H \subset S$, and take $m \in \BN$ large enough such that $H^i(S,I(mH)) = 0$ for all $i \geq 1$ and all colength $n$ ideals $I$ (this choice is possible because flat families of ideals of fixed colength are bounded). Then we let:
%$$
%d = \dim_{\BC} H^0(S,I(mH)) = \chi(S,I(mH))
%$$
%and consider $\CV = p_2^*\CO(-mH)^{\oplus d}$, where $p_1,p_2 : \Hilb_n \times S \rightarrow \Hilb_n, S$ are the projections onto the first and sectond factors, respectively. If $m$ is large enough, the map:
%$$
%\CV \rightarrow \CI \qquad \text{induced by adjunction from} \qquad \CO_{\Hilb_n}^{\oplus d} \rightarrow p_{1*}(\CI \otimes p_2^*\CO(mH)) 
%$$
%is surjective. Letting $\CW$ be the kernel of this map, one can show that $\CW$ is locally free by the argument in Proposition 2.2 of \cite{Shuf}. 

%\end{proof}

\noindent Recall from Subsection \ref{sub:nakajima} the nested Hilbert scheme:
\begin{equation}
\label{eqn:def z1}
\Hilb_{n,n+1} = \Big \{ (I, I') \text{ such that } I \supset_x I' \text{ for some } x \in S \Big \} \subset \Hilb_{n} \times \Hilb_{n+1}
\end{equation}
Above and throughout this Section, we will write $I \supset_x I'$ if $I \supset I'$ and $I/I' \cong \BC_x$. \\

\begin{proposition}
\label{prop:smooth 1}

$\eHilb_{n,n+1}$ is smooth of dimension $2n+2$, and the morphism:
$$
\eHilb_{n,n+1} \xrightarrow{p_S} S, \qquad (I \supset I') \mapsto \emph{supp }I/I'
$$
is smooth. \\

\end{proposition}

\noindent The Proposition above is well-known, except perhaps the fact that $p_S$ is smooth. This fact is easy to show, for example it is proved in \cite{Hecke} by showing that $p_S$ is a submersion. All we will use in the present paper is that $p_S$ is flat. \\

\subsection{} 
\label{sub:nested scheme}

Let us describe the scheme structure of $\Hilb_{n,n+1}$. Consider the maps:
\begin{equation}
\label{eqn:diagram z1}
\xymatrix{& \Hilb_{n,n+1} \ar[ld]_{p_-} \ar[d]^{p_S} \ar[rd]^{p_+} & \\ \Hilb_{n} & S & \Hilb_{n+1}} \qquad \xymatrix{& (I \supset_x I')  \ar[ld]_{p_-} \ar[d]^{p_S} \ar[rd]^{p_+} & \\ I & x & I'}
\end{equation}
and the tautological line bundle on the nested Hilbert scheme:
\begin{equation}
\label{eqn:taut 1}
\xymatrix{\CL \ar@{..>}[d] \\ \Hilb_{n,n+1}} \qquad \CL|_{(I \supset I')} = I/I'
\end{equation}
Throughout the remainder of this paper, we will write $\BP(\CE) = \text{Proj} (\text{Sym} (\CE))$. \\

\begin{proposition}
\label{prop:1 minus}

Let $\CI$ be the universal ideal sheaf on $\eHilb_n \times S$, and let $\CV$, $\CW$ be the vector bundles of Proposition \ref{prop:length 1}. Then we have the commutative diagram:
\begin{equation}
\label{eqn:p minus} 
\xymatrix{\eHilb_{n,n+1} \ar@{^{(}->}[r] \ar[rd]_{p_- \times p_S} & \BP_{\eHilb_n \times S}(\CV) \ar[d]^\rho \\ & \eHilb_n \times S}
\end{equation}
where the horizontal arrow is the zero locus of the following map of vector bundles:
\begin{equation}
\label{eqn:section minus} 
\sigma_- : \rho^*(\CW) \rightarrow \rho^*(\CV) \rightarrow \CO(1)
\end{equation}
on $\BP_{\eHilb_n \times S}(\CV)$. Moreover, $\sigma_-$ is regular (i.e. its Koszul complex is acyclic except in the right-most place) and $\CL$ is isomorphic to the restriction of $\CO(1)$ to $\eHilb_{n,n+1}$. \\

\end{proposition}

\noindent Proposition \ref{prop:1 minus} was proved in \cite{Hecke}, as was the following analogous version with $p_-$ replaced by $p_+$. \\

\begin{proposition}
\label{prop:1 plus}

Let $\CI'$ be the universal ideal sheaf on $\eHilb_{n+1} \times S$, and let $\CV', \CW'$ be the vector bundles of Proposition \ref{prop:length 1}. Then we have the commutative diagram:
\begin{equation}
\label{eqn:p plus} 
\xymatrix{\eHilb_{n,n+1} \ar@{^{(}->}[r] \ar[rd]_{p_+ \times p_S} & \BP_{\eHilb_{n+1} \times S}({\CW'}^\vee \otimes \CK_S) \ar[d]^\rho \\ & \eHilb_{n+1} \times S}
\end{equation}
where the horizontal arrow is the zero locus of the following map of vector bundles:
\begin{equation}
\label{eqn:section plus} 
\sigma_+ : \rho^*({\CV'}^\vee \otimes \CK_S) \rightarrow \rho^*({\CW'}^\vee \otimes \CK_S) \rightarrow \CO(1)
\end{equation}
on $\BP_{\eHilb_{n+1} \times S}({\CW'}^\vee \otimes \CK_S)$. Moreover, $\sigma_+$ is regular and $\CL = \CO(-1)|_{\eHilb_{n,n+1}}$. \\

\end{proposition}

\begin{proof} \emph{of Claim \ref{claim}:} We have the following formulas:
\begin{align}
&(p_+ \times p_S)_*(c_1(\CL)^k) = (-1)^k c_{k+2}(\CI \otimes \CK_S^{-1}) \label{eqn:segre 1} \\
&(p_- \times p_S)_*(c_1(\CL)^k) = (-1)^k c_k(-\CI) \label{eqn:segre 2}
\end{align}		
Let us provide a quick proof for \eqref{eqn:segre 2}, and leave the analogous case of \eqref{eqn:segre 1} as an exercise to the interested reader. Since top Chern classes of vector bundles are equal to zero loci of regular sections, Proposition \ref{prop:1 minus} implies that:
$$
(p_- \times p_S)_*(c_1(\CL)^k) = \rho_* \left(c(\rho^*(\CW^\vee), \CO(1)) \cdot c_1(\CO(1))^k \right)
$$
where $c(\CE,z) = \sum_{k = 0}^r c_k(\CE) z^{r-k}$ denotes the Chern polynomial of an arbitrary rank $r$ vector bundle $\CE$. Since $\rho$ is the projectivization of the vector bundle $\CV$, we have:
$$
\rho_* \left(c_1(\CO(1))^k \right) = c_{k-r+1}(-\CV^\vee) = \text{coefficient of }z^{-1} \text{ in } c(-\CV^\vee,z)\cdot z^k
$$
(by the theory of Segre classes). Combining the two displays above yields precisely:
$$
(p_- \times p_S)_*(c_1(\CL)^k) = \text{coefficient of }z^{-1} \text{ in } c(\CW^\vee-\CV^\vee,z) \cdot z^k
$$
Then $\CI = \CV/\CW$ and the fact that $c_k(-\CI^\vee) = (-1)^k c_k(-\CI)$ imply \eqref{eqn:segre 2}. \\

\noindent Since Claim \ref{claim} is stated in the context of a K3 surface $S$, we will assume $\CK_S \cong \CO_S$ for the remainder of this proof, as this will make our formulas simpler. Let us recall that the Chern character and the total Chern class:
$$
\ch(V) = \sum_{n \geq 0} \ch_n(V), \qquad \qquad c(V) = \sum_{n \geq 0} (-1)^n c_n(V)
$$
are connected by the operations:
\begin{align*} 
&\Psi(\ch(V)) = c(V), \quad \text{where} \quad \Psi \left(\sum_{n \geq 0} a_n \right) = \exp\left(\sum_{n \geq 1} -(n-1)!a_n \right) \\
&\Phi(c(V)) = \ch(V), \quad \text{where} \quad \Phi \left(\sum_{n \geq 0} - \frac {a_n}{(n-1)!} \right) = \log \left(\sum_{n \geq 0} a_n \right)
\end{align*} 
with $a_n$ being a degree $n$ class in the Chow group (the statements above are proved by checking them when $V$ is a line bundle, and then using the fact that $\ch$ is additive and $c$ is multiplicative). We have a short exact sequence on $\Hilb_{n,n+1} \times S$:
\begin{equation}
\label{eqn:ses}
0 \rightarrow p_+^*(\CI) \rightarrow p_-^*(\CI) \rightarrow \CL \otimes (p_S \times \text{Id})^*(\CO_\Delta) \rightarrow 0
\end{equation}
(where $\Delta :  S \hookrightarrow S \times S$ is the diagonal). In the relation above and throughout this computation, we abuse notation and write $\CL$ both for the tautological line bundle on $\Hilb_{n,n+1}$ and for its pull-back to $\Hilb_{n,n+1} \times S$ and $\Hilb_{n,n+1} \times S \times S$. The additivity of Chern character implies the following identity in $A^*(\Hilb_{n,n+1} \times S)$:
$$
\ch(p_\pm^*(\CI)) = \ch(p_\mp^*(\CI)) \mp \left(\sum_{n \geq 0} \frac {c_1(\CL)^n}{n!} \right) \cdot (p_S \times \text{Id})^*\left([\Delta] - \frac {[\Delta]^2}{12} \right)
$$
We may pass this identity through the transformation $\Psi$, and obtain:
\begin{equation}
\label{eqn:fra} 
c(p_\pm^*(\pm \CI)) = c(p_\mp^*(\pm \CI)) \left[1 + (p_S \times \text{Id})^*([\Delta]) \sum_{n=0}^\infty c_1(\CL)^n (n+1) \right] % &c(p_-^*(\CI)) = c(p_+^*(\CI))  (p_S \times \text{Id})^* \left(1 - [\Delta] \sum_{n=0}^\infty c_1(\CL)^n (n+1) + [\Delta]^2 \sum_{n=0}^\infty c_1(\CL)^n {n+3 \choose 3} \right) \label{eqn:fra 2}
\end{equation}
(this fact uses $[\Delta]^3=0$, which follows from $[\Delta]$ being a codimension 2 class on the fourfold $S \times S$). With this in mind, we may perform the following computation:
\begin{multline*}
(p_+ \times p_{S_1} \times \text{Id}_{S_2})_* \circ \fr^a \circ (p_- \times \text{Id}_{S_2})^* \circ (p_- \times p_{S_2})_* (c_1(\CL)^b) \stackrel{\eqref{eqn:segre 2}}{=} \\ = (p_+ \times p_{S_1} \times \text{Id}_{S_2})_* \Big[c_1(\CL)^a \cdot (-1)^b c_b((p_- \times \text{Id}_{S_2})^*(-\CI)) \Big] \stackrel{\eqref{eqn:fra}}{=} \\ = (-1)^b (p_+ \times p_{S_1} \times \text{Id}_{S_2})_* \Big[c_1(\CL)^a c_b(-\CI_2) + [\Delta] \sum_{n=0}^\infty (-1)^{n} c_1(\CL)^{a+n} c_{b-n-2}(-\CI_2)  (n+1)  \Big]  \\
\stackrel{\eqref{eqn:segre 1}}{=} (-1)^{a+b} c_{a+2}(\CI_1) c_b(-\CI_2) + [\Delta] \sum_{n=0}^\infty (-1)^{a+b} c_{a+n+2}(\CI) c_{b-n-2}(-\CI) (n+1)
\end{multline*}
Above, $\CI_1$ and $\CI_2$ are the pull-backs of the universal ideal sheaves from the factors $\Hilb_n \times S_1$ and $\Hilb_n \times S_2$ (respectively) to the product $\Hilb_n \times S_1 \times S_2$. We suppress the indices on $\CI$ in the sum on the last line because $[\Delta] \CI_1 = [\Delta] \CI_2$. Similarly:
\begin{multline*} 
(p_- \times \text{Id}_{S_1} \times p_{S_2})_* \circ \fr^b \circ (p_+ \times \text{Id}_{S_1})^* \circ (p_+ \times p_{S_1})_* (c_1(\CL)^a) \stackrel{\eqref{eqn:segre 1}}{=} \\ = (p_- \times \text{Id}_{S_1} \times p_{S_2})_* \Big[ c_1(\CL)^b \cdot (-1)^a c_{a+2}((p_+ \times \text{Id}_{S_1})^*(\CI)) \Big] \stackrel{\eqref{eqn:fra}}{=} \\ = (-1)^a  (p_- \times \text{Id}_{S_1} \times p_{S_2})_* \Big[c_{a+2}(\CI_1) c_1(\CL)^b + [\Delta] \sum_{n=0}^\infty (-1)^n c_{a-n}(\CI_1) c_1(\CL)^{b+n}(n+1) \Big] \\ \stackrel{\eqref{eqn:segre 2}}{=}  (-1)^{a+b} c_{a+2}(\CI_1) c_b(-\CI_2) + [\Delta] \sum_{n=0}^\infty (-1)^{a+b} c_{a-n}(\CI) c_{b+n}(-\CI) (n+1)
\end{multline*}
Taking the difference between the two equations above yields:
\begin{equation}
\label{eqn:want 5}
\text{LHS of \eqref{eqn:want 4}} = [\Delta] \sum_{n \in \BZ} (-1)^{a+b} c_{a+n+1}(\CI) c_{b-n-1}(-\CI) n
\end{equation}
which we claim is precisely the right-hand side of \eqref{eqn:want 4}. This claim follows from the identity of Chern classes (we assume $a+b>0$ for simplicity, although the case $a+b \leq 0$ is analogous and left to the interested reader):
\begin{multline*}
\sum_{n \in \BZ} (-1)^{a+b} c_{a+n+1}(\CI) c_{b-n-1}(-\CI) n = \text{coefficient of }z^{a+b-1} \text{ in } \frac {d  c(\CI,z)}{dz} c(-\CI,z) - \\
- (a+1)  \cdot \text{coefficient of }z^{a+b} \text{ in } c(\CI,z) c(-\CI,z)= \text{coefficient of }z^{a+b-1} \text{ in } \frac {d \log(c(\CI,z))}{dz} 
\end{multline*}
$ = - (a+b)! \cdot \ch_{a+b}(\CI)$. The latter equality holds because both sides are additive in $\CI$, and it is straightforward to check it when $\CI$ is replaced by a line bundle. Since $[\CI] = 1 - [\CO_\CZ]$ in the $K$--theory group of $\Hilb \times S$, we have that $- \ch_{a+b}(\CI) = \ch_{a+b}(\CO_\CZ)$, hence the right-hand side of \eqref{eqn:want 5} equals the right-hand side of \eqref{eqn:want 4}. 

\end{proof}

%\begin{proof} The proof is very similar with that of Proposition \ref{prop:1 minus}. Consider the fiber of $\rho$ above a certain $(I',x) \in \Hilb_{n+1} \times S$, and denote by $0 \rightarrow W' \rightarrow V' \rightarrow I' \rightarrow 0$ the restriction of the short exact sequence \eqref{eqn:length 1} to $\{I'\} \times S$. A point in the fiber $\rho^{-1}(I',x)$ which lies in the zero locus of $\sigma_+$ is the same thing as an element of:
%$$
%\Ker \Big(\Hom({W'}|_x \otimes \CK^{-1}_S|_x, \BC)^\vee \rightarrow \Hom({V'}|_x \otimes \CK^{-1}_S|_x, \BC)^\vee \Big)
%$$
%up to constant multiple. By adjunction, this is the same as an element of:
%$$
%\Ker \Big(\Hom(W' \otimes \CK^{-1}_S, \BC_x)^\vee \rightarrow \Hom(V' \otimes \CK^{-1}_S, \BC_x)^\vee \Big)
%$$
%up to constant multiple. By Serre duality, this is the same as an element of:
%$$
%\Ker \Big(\Ext^2(\BC_x, W') \rightarrow \Ext^2(\BC_x, V') \Big)
%$$
%up to constant multiple. Because of the long exact sequence corresponding to $0 \rightarrow W' \rightarrow V' \rightarrow I' \rightarrow 0$, the kernel above is simply $\Ext^1(\BC_x,I')$ (this uses the well-known fact that $\Ext^1(Q,F) = 0$ if $F$ is locally free and $Q$ is a length 0 sheaf on a smooth surface), and this precisely corresponds to a point in $(p_+ \times p_S)^*(I',x)$. 

%\end{proof}

\subsection{} 

Let us consider the following more complicated cousin of the scheme \eqref{eqn:def z1}:
\begin{multline}
\Hilb_{n-1,n,n+1} = \Big \{ (I, I',I'') \text{ such that } I \supset_x I' \supset_x I'' \\
\text{for some } x \in S \Big \} \subset \Hilb_{n-1} \times \Hilb_{n} \times \Hilb_{n+1} \label{eqn:punctual}
\end{multline}
The following result was proved in \cite{W surf}, in the analogous setup of moduli spaces of stable sheaves, but the modifications to the case of Hilbert schemes are minimal. \\

%and the map:
%\begin{equation}
%\label{eqn:almost smooth map}
%\Hilb_{n,n+1,n+2} \xrightarrow{p_{S \times S}} S \times S
%\end{equation}
%that records the support points of the inclusions $I \supset I'$ and $I' \supset I''$. We let:
%$$
%\Hilb^\bullet_{n,n+1,n+2} = p_{S \times S}^{-1}(\Delta_S)
%$$
%be the locus of triples of nested ideals where the support points coincide. It is easy to observe that both $\Hilb_{n,n+1,n+2}$ and $\Hilb_{n,n+1,n+2}^\bullet$ are projective. \\

\begin{proposition}
	\label{prop:smooth 2}
	
$\eHilb_{n-1,n,n+1}$ is smooth of dimension $2n+1$. \\
	
\end{proposition}

\noindent Note that the scheme \eqref{eqn:punctual} is endowed with line bundles $\CL_1$ and $\CL_2$:
\begin{equation}
\label{eqn:taut 2}
\xymatrix{\CL_1, \CL_2 \ar@{..>}[d] \\ \Hilb_{n-1,n,n+1}} \qquad \CL_1|_{(I \supset I' \supset I'')} = I'/I'', \qquad \CL_2|_{(I \supset I' \supset I'')} = I/I'
\end{equation}
Consider also the natural maps which forget either $I''$ or $I$:
\begin{equation}
\label{eqn:diagram z2}
\xymatrix{& \Hilb_{n-1,n,n+1} \ar[ld]_{\pi_-} \ar[d]^{\pi_+} \\ \Hilb_{n-1,n} & \Hilb_{n,n+1}}
 \qquad \xymatrix{& (I \supset I' \supset I'') \ar[ld]_{\pi_-} \ar[d]^{\pi_+} \\ (I \supset I') & (I' \supset I'')}\end{equation}
Let $\Gamma : \Hilb_{n,n+1} \hookrightarrow \Hilb_{n,n+1} \times S$ denote the graph of the map $p_S$. \\ %We will abuse notation by writing $\CI \cong \CV/\CW$ and $\CI' \cong \CV'/\CW'$ for the universal ideal sheaves on $\Hilb_{n,n+1} \times S$ that parametrize the ideals of colength $n$ and $n+1$, respectively, so we have the following universal short exact sequence of sheaves on $\Hilb_{n,n+1} \times S$:
%\begin{equation}
%\label{eqn:exact 1}
%0 \rightarrow \CI' \rightarrow \CI \rightarrow \Gamma_*(\CL) \rightarrow 0
%\end{equation}
%All sheaves in \eqref{eqn:exact 1} are flat over $\Hilb_{n,n+1}$. The following was proved in \cite{Hecke}: \\

\begin{proposition}
\label{prop:2 minus}

(\cite{Hecke}) Let $\CI$ denote the universal ideal sheaf on $\eHilb_n \times S$. Then:
$$
\xymatrix{\eHilb_{n-1,n,n+1} \ar@{^{(}->}[r] \ar[rd]_{\pi_-} & \BP_{\eHilb_{n-1,n}}(\Gamma^*(\CV)) \ar[d]^{\rho_-} \\ & \eHilb_{n-1,n}} 
$$
$$
\xymatrix{\eHilb_{n-1,n,n+1} \ar@{^{(}->}[r] \ar[rd]_{\pi_+} & \BP_{\eHilb_{n,n+1}}(\Gamma^*({\CW}^\vee \otimes \CK_S)) \ar[d]^{\rho_+} \\ & \eHilb_{n,n+1}}
$$
where the horizontal arrows are the zero loci of the following maps of vector bundles:
\begin{align}
&\sigma'_- : \rho_-^* \left( \frac {\Gamma^*(\CW)}{\CL \otimes \CK_S} \right) \xrightarrow{\text{induced by }\sigma_-} \CO(1) \label{eqn:section 2 minus}  \\
&\sigma'_+ : \rho_+^* \left( \frac {\Gamma^*({\CV}^\vee \otimes \CK_S)}{\CL^{-1} \otimes \CK_S} \right) \xrightarrow{\text{induced by }\sigma_+} \CO(1) \label{eqn:section 2 plus} 
\end{align}
on $\BP_{\eHilb_{n-1,n}}(\Gamma^*(\CV))$ and $\BP_{\eHilb_{n,n+1}}(\Gamma^*({\CW}^\vee \otimes \CK))$, respectively (above, $\sigma_-$ and $\sigma_+$ denote the sections given by the same formulas as \eqref{eqn:section minus} and \eqref{eqn:section plus}, respectively). \\

%where $\CL \otimes \CK_S \hookrightarrow \Gamma^*(\CW')$ due to the presentation:
%$$
%\eHilb_{n,n+1} \hookrightarrow \BP_{\eHilb_{n+1} \times S}({\CW'}^\vee \otimes \CK_S)
%$$
%from \eqref{eqn:p plus}. The section $\sigma'_-$ is regular, and $\CL_1 = \CO(1)|_{\eHilb^\bullet_{n,n+1,n+2}}$, $\CL_2 = \pi_-^*(\CL)$. \\

\end{proposition}

\noindent Moreover, the line bundles $\CL_1$ and $\CL_2$ are isomorphic to the restrictions to $\Hilb_{n-1,n,n+1}$ of the tautological line bundles $\CO_{\BP_{\Hilb_{n-1,n}}}(1)$ and $\CO_{\BP_{\Hilb_{n,n+1}}}(-1)$, respectively. Therefore, the definition of Chern/Segre classes implies that:
\begin{align}
&\pi_{+*}(c_1(\CL_2)^k) = (-1)^k c_{k+1} \left(\CI \otimes \CK_S^{-1} - \CL \otimes \CK_S^{-1} \right) \label{eqn:segre 4} \\
&\pi_{-*}(c_1(\CL_1)^k) = (-1)^{k-1}  c_{k-1}\left(- \CI - \CL \otimes \CK_S \right) \label{eqn:segre 3} 
\end{align}

\subsection{} Suppose we have a fiber square of schemes with all maps being proper:
\begin{equation}
\label{eqn:fiber square 0}
\xymatrix{Y' \ar[r]^{\iota'} \ar[d]_{\eta'} & X' \ar[d]^\eta \\ Y \ar[r]^\iota & X}
\end{equation}
and we assume that the map $\iota$ is a regular embedding, cut out by a section $\sigma : \CO_X \rightarrow V$ of a vector bundle $V$ on $X$. It is well-known that if $X$ and $X'$ are Cohen-Macaulay, then the fiber square is derived (i.e. $\iota'$ is a regular embedding cut out by the section $\eta^*(\sigma) : \CO_{X'} \rightarrow \eta^*(V)$) if and only if:
\begin{equation}
\label{eqn:equality}
\dim X' - \dim Y' = \dim X - \dim Y
\end{equation}
On the other hand, suppose $\eta^*(\sigma)$ lands in the kernel of a map $\eta^*(V) \rightarrow E$, where $E$ is a rank $r$ vector bundle on $X'$. Then the embedding $\iota'$ is regular, and cut out by the induced map $\eta^*(\sigma) : \CO_{X'} \rightarrow \text{Ker }(\eta^*(V) \rightarrow E)$, if and only if:
\begin{equation}
\label{eqn:excess equality}
\dim X' - \dim Y' = \dim X - \dim Y + r
\end{equation}
We will refer to this situation as \emph{excess intersection}, and call $E$ the \emph{excess bundle}. Then the following formulas are well-known (\cite{F}): \\

\begin{lemma}
\label{lem:intersection}

If the square \eqref{eqn:fiber square 0} is derived (i.e. the setup of \eqref{eqn:equality} holds), then we have the following equality of morphisms of Chow groups:
\begin{equation}
\label{eqn:base change}
\eta'_* \circ {\iota'}^* = \iota^* \circ \eta_* 
\end{equation}
If we are in the excess intersection situation (i.e. the setup of \eqref{eqn:excess equality} holds), then we have the following equality of morphisms of Chow groups:
\begin{equation}
\label{eqn:excess base change}
\eta'_* \circ (e(E) \cdot {\iota'}^*) = \iota^* \circ \eta_* 
\end{equation}
where $e$ denotes the Euler class (or top Chern class) of the vector bundle $E$. \\

\end{lemma}

\subsection{} Consider the following diagram, obtained by combining \eqref{eqn:diagram z1} with \eqref{eqn:diagram z2}:
\begin{equation}
\label{eqn:fiber square}
\xymatrix{& \Hilb_{n-1,n,n+1} \ar[ld]_{\pi_-} \ar[rd]^{\pi_+} & \\ \Hilb_{n-1,n} \ar[rd]_{p_+ \times p_S} & & \Hilb_{n,n+1} \ar[ld]^{p_- \times p_S} \\
& \Hilb_{n} \times S &}
\end{equation}
This diagram is a fiber square, as can be observed by recalling the definitions of the nested Hilbert schemes involved as answers to moduli problems. It is not a derived fiber square, which can be observed by comparing \eqref{eqn:section minus} with \eqref{eqn:section 2 minus}. However, it is an instance of excess intersection \eqref{eqn:excess equality} (see \cite{Hecke} for a proof), hence we obtain the following special case of \eqref{eqn:excess base change}: \\

\begin{proposition}
\label{prop:excess}

We have the following equality of maps between Chow groups:
\begin{equation}
\label{eqn:excess}
\pi_{+*} \circ \Big[(l_1 - l_2 - p_S^*(t)) \cdot \pi_-^* \Big] = (p_- \times p_S)^* \circ (p_+ \times p_S)_*
\end{equation}
where $l_1 = c_1(\CL_1)$, $l_2 = c_1(\CL_2)$, $t = c_1(\CK_S)$ are classes in the Chow groups of all spaces in \eqref{eqn:fiber square}. The analogous result holds with the roles of $+$ and $-$ switched. \\

\end{proposition}

\noindent Indeed, the only thing to note is that the excess bundle is $\CL_1 \otimes \CL_2^{-1} \otimes p_S^*(\CK_S^{-1})$ (which arises from the dual of the denominator of \eqref{eqn:section 2 plus}). The fact that \eqref{eqn:excess equality} holds with $r=1$ is a consequence of Propositions \ref{prop:smooth 1} and \ref{prop:smooth 2}. \\

%\begin{proof} We will use the notation $I'',I',I$ for the ideals of colength $n+2,n+1,n$ parametrized by the spaces in diagram \eqref{eqn:fiber square}, and we will write $\CL_1$ and $\CL_2$ for the line bundles that parametrize the quotients $I'/I''$ and $I/I'$, respectively. Let $\CV'$, $\CW'$ denote the locally free sheaves that Proposition \ref{prop:length 1}  associates to the universal sheaf $\CI'$. According to Proposition \ref{prop:1 minus} (respectively \ref{prop:2 minus}), the space $\Hilb_{n+1,n+2}$ (respectively $\Hilb_{n,n+1,n+2}$) is cut out by a regular section of the vector bundle:
%$$
%\left( \CW' \right)^\vee \otimes \CL_1, \qquad \text{respectively } \left(  \frac {\CW'}{\CL_2 \otimes p_S^*(\CK_S)} \right)^\vee \otimes \CL_1
%$$
%inside the projectivization of the vector bundle $\CV'$ on the space $\Hilb_{n+1} \times S$ (respectively $\Hilb_{n,n+1}$). Therefore, the square \eqref{eqn:fiber square} is a particular case of the excess intersection situation \eqref{eqn:excess equality}, with excess line bundle:
%$$
%E = \CL_1 \otimes \CL_2^{-1} \otimes p_S^*(\CK_S^{-1})
%$$
%\end{proof}

\subsection{} The Hilbert scheme $\Hilb_n$ has dimension $2n$. If we fix a closed point $x \in S$, then we may define the \emph{defect} of an ideal $I \subset \CO_S$ at the point $x$: this is simply the length at $x$ of the finite length sheaf $\CO_S/I$. We have the locally closed stratificaton:
\begin{equation}
\label{eqn:strata}
\Hilb = \bigsqcup_{d=0}^\infty \Hilb_{\text{def }d}
\end{equation}
by defect at the chosen point $x \in S$. It is well-known that $\Hilb_{\text{def }0}$ is open, while:
\begin{equation}
\label{eqn:codim}
\text{codim } \Hilb_{\text{def }d} = d + 1
\end{equation}
for any $d>0$ (see, for example, Lemma 6.10 of \cite{Nak}). \\

\begin{proposition}
\label{prop:generators}

Consider any finite colength ideal $I \subset \CO_S$ and any closed point $x \in S$. Then we have the following estimate:
\begin{equation}
\label{eqn:generators}
\dim_{\BC} \emph{Hom}(I,\BC_x) - 1 \leq \sqrt{2n + \frac 14} - \frac 12
\end{equation}
where $n$ is the colength of $I$ at the point $x$. When $n$ is small, we actually have:
\begin{align}
&\text{if }n=0, \text{ then } \dim_{\BC} \emph{Hom}(I,\BC_x) - 1 = 0 \label{eqn:generators 0} \\
&\text{if }n=1, \text{ then } \dim_{\BC} \emph{Hom}(I,\BC_x) - 1 = 1 \label{eqn:generators 1} \\
&\text{if }n=2, \text{ then } \dim_{\BC} \emph{Hom}(I,\BC_x) - 1 = 1 \label{eqn:generators 2} \\
&\text{if }n=3, \text{ then } \dim_{\BC} \emph{Hom}(I,\BC_x) - 1 = 1 \text{ or }2 \label{eqn:generators 3}
\end{align}
where in \eqref{eqn:generators 3}, the value $2$ is taken on a positive codimension locus of ideals. \\
	
\end{proposition}

\begin{proof} \emph{of Proposition \ref{prop:generators}:} The problem is purely local, so we may assume that $S = \BA^2$ and $x = (0,0)$. In this case, the torus action $T = \BC^* \times \BC^* \curvearrowright \BA^2$ extends to the projective variety $\Hilb^\bullet_n \subset \Hilb_n$ parametrizing length $n$ subschemes of $\BA^2$ supported at $(0,0)$. It is well-known that the $T$--fixed points of this action are:
$$
I_\lambda = (x^{\lambda_1}, x^{\lambda_2}y,...,x^{\lambda_t} y^{t-1}) \subset \BC[x,y]
$$
as $\lambda = (\lambda_1 \geq \lambda_2 \geq ... \geq \lambda_t)$ goes over all partitions of $n$. Because the dimension \eqref{eqn:generators} is upper semicontinuous in $I$, it is enough to prove \eqref{eqn:generators} when $I = I_\lambda$ for some partition $\lambda$. In this case, it is easy to see that:
$$
\dim_{\BC} \Hom(I_\lambda,\BC_x) - 1 = \# \text{ number of different parts of } \lambda
$$
If we assume that $\lambda$ consists of the distinct natural numbers $n_1,...,n_s$ with multiplicities $m_1,...,m_s \in \BN$, then the inequality \eqref{eqn:generators} is a consequence of:
$$
s = \sqrt{2 (1+2+...+s) + \frac 14} - \frac 12 \leq \sqrt{2(n_1m_1+...+n_sm_s) + \frac 14} - \frac 12 =  \sqrt{2n + \frac 14} - \frac 12
$$
Formula \eqref{eqn:generators 0} is trivial. To establish \eqref{eqn:generators 1}--\eqref{eqn:generators 3}, one observes that for $n \in \{1,2\}$ all ideals $I \in \Hilb_n$ are curvilinear near $x$, i.e. contain the ideal $J$ of a smooth curve. When $n=3$ almost all ideals $I \in \Hilb_n$ are curvilinear near $x$, except the square of the maximal ideal of the closed point $x$, which leads to the value 2 in \eqref{eqn:generators 3}. It is easy to prove that an ideal $I$ which is curvilinear near a point $x \in S$ has the property that $\dim_{\BC} \Hom(I,\BC_x) - 1 = 1$, and this implies \eqref{eqn:generators 1}--\eqref{eqn:generators 3}. In general, on the irreducible variety $\Hilb_n^\bullet$, curvilinear ideals form a dense open set.
	
\end{proof}

\subsection{}

The following is our main geometric computation (see \cite{W surf} for an analogous version in the context of the $K$--theory of moduli spaces of stable sheaves): \\

\begin{lemma}
\label{lem:main}

Consider the schemes $\eHilb_{n,n+1} = \{(I_0 \subset I_1)\}$ as well as:
\begin{equation}
\label{eqn:var 1}
\eHilb_{n,n+1} \underset{\eHilb_{n} \times S}{\times} \eHilb_{n,n+1} = \{(I_0 \subset I_1 \supset I_0')\}
\end{equation}
\begin{equation}
\label{eqn:var 2}
\eHilb_{n-1,n,n+1} \underset{\eHilb_{n-1,n}}{\times} \eHilb_{n-1,n,n+1} = \{(I_0 \subset I_1 \supset I_0', I_1 \subset I_2)\}
\end{equation}
where all the inclusions are required to be supported at the same closed point, henceforth denoted by $x$, which is allowed to vary over $S$. We have the natural maps:
$$
\xymatrix{\eHilb_{n,n+1} \ar[d]^-\delta \\ \eHilb_{n,n+1} \underset{\eHilb_{n} \times S}{\times} \eHilb_{n,n+1}} 
$$
given by $\delta(I_0 \subset I_1) = (I_0 \subset I_1 \supset I_0)$, and:
$$
\xymatrix{\eHilb_{n-1,n,n+1} \underset{\eHilb_{n-1,n}}{\times} \eHilb_{n-1,n,n+1} \ar[d]^-\e \\ \eHilb_{n,n+1} \underset{\eHilb_{n} \times S}{\times} \eHilb_{n,n+1}}
$$
given by forgetting $I_2$. Then we have the formulas:
\begin{equation}
\label{eqn:formula 1}
\e_*(1) + \delta_*(1) = 1
\end{equation}
\begin{equation}
\label{eqn:formula 2}
\e_*(l_2) + \delta_*(l) =  - t \cdot \e_*(1) + l_1 + l_1'
\end{equation}
where $l_1, l_1', l_2$ are the first Chern classes of the line bundles which keep track of the quotients denoted by $I_1/I_0$, $I_1/I_0'$, $I_2/I_1$ in the diagrams above, and $l = l_1|_\delta = l_2|_\delta$. The class $t$ denotes $c_1(\CK_S)$, pulled back from $A^*(S)$ to the Chow groups of the various moduli spaces above via the map that remembers the support point $x$. \\

\end{lemma}

\begin{proof} As we have observed in Proposition \ref{prop:smooth 1}, $\Hilb_{n,n+1}$ is smooth of dimension $2n+2$. It is also connected, as the natural map $p_- \times p_S : \Hilb_{n,n+1} \rightarrow \Hilb_n \times S$ has all geometric fibers isomorphic to projective spaces, see Proposition \ref{prop:1 minus}. We will prove that the variety \eqref{eqn:var 1} has dimension $2n+2$, and has two irreducible components of top dimension, by stratifying it according to the defect of the ideal $I_1$ at the point $x$: \\

\begin{enumerate}

\item if $I_1$ is locally free at the point $x$ (i.e. has defect 0), then: \\

\begin{itemize}

\item $I_1$ contributes $2n$ to the dimension \\

\item $x$ contributes $2$ to the dimension \\

\item $I_0$, $I_0'$ each contributes $\dim \Hom(I_1,\BC_x) - 1 = 0$ to the dimension \\ 

\end{itemize}

\item if $I_1$ has defect $d>0$ at the point $x$, then: \\

\begin{itemize}

\item $I_1$ contributes $2n-1-d$ to the dimension, by \eqref{eqn:codim} \\

\item $x$ contributes $2$ to the dimension \\

\item $I_0$, $I_0'$ each contributes $\dim \Hom(I_1,\BC_x) - 1$ to the dimension \\ 

\end{itemize}

\end{enumerate}

\noindent The stratum (1) has dimension $2n+2$. Similarly, the dimension of stratum (2) when $d=1$ is also $2n+2$. When $d=2$, Proposition \ref{prop:generators} implies that the dimension of stratum (2) is:
$$
2n - 1 + 2 \cdot 1 < 2n+2
$$
while if $d = 3$ its dimension is:
$$
2n - 2 - (0 \text{ or }\underline{1}) + 2 \cdot (1 \text{ or }2) < 2n + 2
$$
The explanation for the word ``or" is that, as we have seen in \eqref{eqn:generators 3}, a colength 3 ideal $I_1$ having $\dim_{\BC} \Hom(I_1,\BC_x) - 1 = 2$ is a positive codimension property, hence the underlined $-1$ appearing in the left-hand side above. Finally, if $d \geq 4$, the dimension of stratum (2) may be estimated using \eqref{eqn:generators}:
$$
\leq 2n - 1 - d + 2 + 2 \left( \sqrt{2d+ \frac 14} - \frac 12 \right) < 2n + 2
$$
We conclude that \eqref{eqn:var 1} has dimension $2n+2$, with two irreducible components of top dimension: one is the closure of the locus where $I_1$ has no defect at $x$, and the other is the closure of the locus where $I_1$ has defect 1 at $x$. Therefore, the square:
\begin{equation}
\label{eqn:fiber square 1}
\xymatrix{& \text{variety } \eqref{eqn:var 1} \ar[ld] \ar[rd] & \\ \Hilb_{n,n+1} \ar[rd] & & \Hilb_{n,n+1} \ar[ld] \\
& \Hilb_{n} \times S&}
\end{equation}
consists only of varieties of dimension $2n+2$. Since the maps on the bottom are local complete intersection morphisms (see Proposition \ref{prop:1 minus}), we conclude that \eqref{eqn:equality} applies and the fiber square \eqref{eqn:fiber square 1} is derived. Therefore, the variety \eqref{eqn:var 1} is l.c.i., hence only has two irreducible components. Similarly, we claim the variety \eqref{eqn:var 2} is irreducible of dimension $2n+2$. Indeed, we can show that its dimension is $\leq 2n + 2$ by considering the stratification according to the defect of the ideal $I_2$ at the point $x$: \\

\begin{enumerate}

\item if $I_2$ is locally free at the point $x$, then: \\

\begin{itemize}

\item $I_2$ contributes $2n-2$ to the dimension \\

\item $x$ contributes $2$ to the dimension \\

\item $I_1$ contributes $\dim \Hom(I_2,\BC_x) - 1 = 0$ to the dimension \\

\item $I_0$, $I_0'$ each contributes $\dim \Hom(I_1,\BC_x) - 1 = 1$ to the dimension \\

\end{itemize}

\item if $I_2$ has colength $d > 0$ at the point $x$, then: \\

\begin{itemize}

\item $I_2$ contributes $2n-3-d$ to the dimension, by \eqref{eqn:codim}  \\

\item $x$ contributes $2$ to the dimension \\

\item $I_1$ contributes $\dim \Hom(I_2,\BC_x) - 1$ to the dimension \\

\item $I_0$, $I_0'$ each contributes $\dim \Hom(I_1,\BC_x) - 1$ to the dimension \\

\end{itemize}

\end{enumerate}

\noindent The dimension of stratum (1) is precisely $2n+2$, and it clearly has a single irreducible component of this dimension. In case (2), we may use Proposition \ref{prop:generators} to obtain that the dimension of the stratum with $d=1$ is:
$$
\leq 2n - 2 + 1 + 2 < 2n + 2
$$
while the dimension of the stratum with $d=2$ is:
$$
\leq 2n - 3 + (0 \text{ or }1) + 2\cdot (2 \text{ or } 1) < 2n + 2
$$
(the explanation for the ``or" is that a colength 3 ideal $I_1$ having $\dim_{\BC}\Hom(I_1,\BC) - 1 = 2$ is a positive codimension property). Finally, the dimension of the stratum (2) with $d \geq 3$ is:
$$
\leq 2n - 1 - d + \underline{\sqrt{2d + \frac 14} - \frac 12} - \boxed{1} + 2 \left( \sqrt{2(d+1) + \frac 14} - \frac 12 \right) < 2n + 2
$$
(the boxed $1$ was subtracted because on the dense open locus of ideals $I_2$ which are curvilinear at $x$, we may replace the underlined term with $1$ in the formula above). Consider the fiber square:
\begin{equation}
\label{eqn:fiber square 2}
\xymatrix{& \text{variety } \eqref{eqn:var 2} \ar[ld] \ar[rd] & \\ \Hilb_{n-1,n,n+1} \ar[rd] & & \Hilb_{n-1,n,n+1} \ar[ld] \\
& \Hilb_{n-1,n}&}
\end{equation}
The dimension of the four spaces in the diagram above are, from top to bottom, $2n+2$, $2n+1$, $2n+1$, $2n$. We claim that the fiber square above is derived, which follows from equality \eqref{eqn:equality} applied to the square \eqref{eqn:fiber square 2}, and the bottom-most maps being l.c.i. morphisms, due to Proposition \ref{prop:2 minus}. We conclude that the variety \eqref{eqn:var 2} is a local complete intersection, hence irreducible of dimension $2n+2$. \\

\noindent Formula \eqref{eqn:formula 1} is simply an equality on the top dimensional irreducible components, and it follows from the fact that the two irreducible components of \eqref{eqn:var 1} are each mapped onto by $\Hilb_{n,n+1}$ and the variety \eqref{eqn:var 2}, respectively, under the maps $\delta$ and $\e$, respectively. As for formula \eqref{eqn:formula 2}, let us consider the fiber square:
$$
\xymatrix{\text{variety \eqref{eqn:var 2}} \ar[d]_\e \ar[r]^a & \Hilb_{n-1,n,n+1} \ar[d]^{\pi_+} \\
\text{variety \eqref{eqn:var 1}} \ar[r]^b & \Hilb_{n,n+1}}
$$
where $a$ and $b$ denote the projections onto the right Cartesian product factor in \eqref{eqn:var 1} and \eqref{eqn:var 2}. Because the dimensions of the varieties above are $2n+2$, except for that of $\Hilb_{n-1,n,n+1}$ which is $2n+1$, the excess intersection formula (as in Proposition \ref{prop:excess}) implies the following equality of morphisms:
$$
\e_* \circ \left( (l_1-l_2-t) \cdot a^* \right) = b^* \circ \pi_{+*}
$$
Applying this equality to the fundamental class gives us:
\begin{equation}
\label{eqn:ja 1}
\e_*(l_2) = \e_*(l_1-t) - b^*(\pi_{+*}(1))
\end{equation}
Because the divisor class $l_1-t$ is pulled back from the variety \eqref{eqn:var 1}, we have: 
\begin{multline}
\e_*(l_1-t) = (l_1-t)\e_*(1) \stackrel{\eqref{eqn:formula 1}}= \\ = l_1(1-\delta_*(1)) - t\e_*(1) = l_1 - \delta_*(l) - t\e_*(1) \label{eqn:ja 2}
\end{multline}
Moreover, as a consequence of \eqref{eqn:segre 4}, we have $\pi_{+*}(1) = - l$, and therefore:
\begin{equation}
\label{eqn:ja 3}
b^*(\pi_{+*}(1)) = b^*(-l) = -l_1'
\end{equation}
Formulas \eqref{eqn:ja 1}, \eqref{eqn:ja 2} and \eqref{eqn:ja 3} yield:
$$
\e_*(l_2) = l_1 - \delta_*(l) - t \e_*(1) + l_1'
$$
which proves \eqref{eqn:formula 2}. 

\end{proof} 

\subsection{}

We will now use the computations in Lemma \ref{lem:main} to obtain certain equalities between the Nakajima operators $\fq_k$ of \eqref{eqn:nakajima}, thus leading to Theorem \ref{thm:lehn}. However, the correspondence \eqref{eqn:diagram zk} has the disadvantage that it is rather badly behaved, and it is hard to use it in order to explicitly compute the operators $\fq_k$ in terms of tautological classes. Therefore, we find it more convenient to factor the Nakajima operators in terms of the nested Hilbert schemes of Subsection \ref{sub:nested}: \\

\begin{theorem}
\label{thm:nakajima}

Consider the operators:
\begin{align*}
&e_\uparrow : \bigoplus_{n=0}^\infty A^*(\eHilb_n) \longrightarrow \bigoplus_{n=0}^\infty A^*(\eHilb_{n,n+1}), & e_\uparrow = p_-^* \\
&e_\downarrow : \bigoplus_{n=0}^\infty A^*(\eHilb_{n,n+1}) \longrightarrow \bigoplus_{n=0}^\infty A^*(\eHilb_{n+1} \times S), & e_\downarrow = (p_+ \times p_S)_* \\
&e_\rightarrow : \bigoplus_{n=1}^\infty A^*(\eHilb_{n-1,n}) \longrightarrow \bigoplus_{n=0}^\infty A^*(\eHilb_{n,n+1}), & e_{\rightarrow} = \pi_{+*} \pi_-^* \\
&f_\uparrow : \bigoplus_{n=0}^\infty A^*(\eHilb_{n}) \longrightarrow \bigoplus_{n=0}^\infty A^*(\eHilb_{n-1,n}), & f_\uparrow = p_+^* \\
&f_\downarrow : \bigoplus_{n=0}^\infty A^*(\eHilb_{n,n+1}) \longrightarrow \bigoplus_{n=0}^\infty A^*(\eHilb_{n} \times S), & f_\downarrow = - (p_- \times p_S)_* \\
&f_\leftarrow : \bigoplus_{n=0}^\infty A^*(\eHilb_{n,n+1}) \longrightarrow \bigoplus_{n=1}^\infty A^*(\eHilb_{n-1,n}), & f_{\leftarrow} = - \pi_{-*} \pi_+^*
\end{align*}
with the maps $p_\pm$ and $\pi_\pm$ as in \eqref{eqn:diagram z1} and \eqref{eqn:diagram z2}. Then we have:
\begin{align}
\fq_k &= e_\downarrow \circ \underbrace{e_\rightarrow \circ ... \circ e_\rightarrow}_{k-1 \text{ operators}} \circ e^\uparrow \label{eqn:factor alpha 1} \\
\fq_{-k} & = f_\downarrow \circ \underbrace{f_\leftarrow \circ ... \circ f_\leftarrow}_{k-1 \text{ operators}} \circ f^\uparrow \label{eqn:factor alpha 2}
\end{align}

\end{theorem}

\begin{proof} We will prove \eqref{eqn:factor alpha 1}, as \eqref{eqn:factor alpha 2} is deduced from it by transposition. The desired formula is an equality of top-dimensional cycles on $\Hilb_{n,n+k}$. Since this variety has a single irreducible component of top dimension (see \cite{Nak}), the formula boils down to proving that for a generic point:
$$
(I, I') \in \Hilb_{n,n+k}
$$
(with $\text{supp } I/I' = \{x\}$) there is a unique way to complete it to a full flag: 
\begin{equation}
\label{eqn:full flag}
I' = I_{n+k} \subset I_{n+k-1} \subset ... \subset I_{n+1} \subset I_n = I
\end{equation}
The reason for this is that the generic point of $\Hilb_{n,n+k}$ is curvilinear, i.e. the quotient $I/I'$ is a quotient of $\CO_C$ for a smooth curve $C \subset S$, and in this case the only choice for the flag \eqref{eqn:full flag} is given by the powers of $\fm_x \subset \CO_C$. 

\end{proof}

\subsection{} 

As a consequence of formulas \eqref{eqn:factor alpha 1} and \eqref{eqn:factor alpha 2}, we have the following result: \\

\begin{proposition}
\label{prop:preserve}

For any $k \in \BZ \backslash 0$ and any $\gamma \in R(S)$, the map $\fq_k(\gamma)$ preserves $A_{\emph{big}}(\eHilb)$. Similarly, for any $k,k' \in \BZ \backslash 0$, the map $\fq_k \fq_{k'}(\Delta)$ preserves $A_{\emph{big}}(\eHilb)$. \\

\end{proposition}

\noindent Strictly speaking, the notion of big tautological classes was only defined for a K3 surface $S$, but the Proposition above holds for any surface $S$, as long as $R(S)$ that appears in Definition \ref{def:a big} is replaced by a subring of $A^*(S)$ that contains the Chern classes of the tangent bundle. \\

\begin{proof} Let us first prove the statement about $\fq_k(\gamma)$, assuming $k>0$ (the case $k<0$ is analogous). Recall that $A^*_\big(\Hilb_n) \subset A^*(\Hilb_n)$ is the subring generated by:
\begin{equation}
\label{eqn:classes}
\pi_{1*} \left[ \ch_{k_1} \left( \frac {\CO_{\Hilb_n \times S}}{\CI} \right)  ... \ch_{k_t} \left( \frac {\CO_{\Hilb_n \times S}}{\CI} \right)\cdot \pi_2^*(\gamma) \right]
\end{equation}
for any $k_1,...,k_t>0$ and any $\gamma \in R(S) \subset A^*(S)$. In a similar vein, let: 
\begin{equation}
\label{eqn:fo}
A^*_\big(\Hilb_n \times S) \subset A^*(\Hilb_n \times S)
\end{equation}
denote the subring generated by the pull-backs of classes \eqref{eqn:classes} from $\Hilb_n$, the pull-back of classes in $R(S)$ from $S$, and the Chern character of $\CI$ itself. Let:
$$
A^*_\big(\Hilb_{n,n+1}) \subset A^*(\Hilb_{n,n+1})
$$
denote the subring generated by $c_1(\CL)$, the classes $p_S^*(R(S))$, the pull-backs of classes \eqref{eqn:classes} from either $\Hilb_n$ or $\Hilb_{n+1}$, and arbitrary Chern classes of $\Gamma^*(\CI)$ and $\Gamma^*(\CI')$ (where $\CI$ and $\CI'$ are the two tautological ideal sheaves on the space $\Hilb_{n,n+1} \times S$, and $\Gamma : \Hilb_{n,n+1} \rightarrow \Hilb_{n,n+1} \times S$ is the graph of the map $p_S$). Let:
$$
A^*_\big(\Hilb_{n,n+1,n+2}) \subset A^*(\Hilb_{n,n+1,n+2})
$$
be defined analogously, with respect to all possible ideal sheaves on $\Hilb_{n,n+1,n+2}$. To show that $\fq_k(\gamma)$ preserves the ring of big tautological classes, it suffices by \eqref{eqn:factor alpha 1} to show that the maps $p_-^*, \pi_-^*, \pi_{+*}, (p_+ \times p_S)_*$ send $A^*_\big(...)$ to $A^*_\big(...)$. This is obvious for the pull-back maps by the very definitions of the various rings above, so we only need to prove it for the push-forwards. For example, we must show that:
\begin{equation}
\label{eqn:1}
(p_+ \times p_S)_* \left( \prod_{k_i,\gamma} \pi_{1*} \left[ \prod_i \ch_{k_i} \left( \frac {\CO_{\Hilb_{n,n+1} \times S}}{\CI} \right) \cdot \pi_2^*(\gamma) \right] \cdot \prod_{j} \ch_{k'_j}(\Gamma^*(\CI)) \cdot c_1(\CL)^d \right)
\end{equation}
lies in $A^*_\big(\Hilb_{n+1} \times S)$, for any choice of $k_i,k'_j,d,\gamma$ (there is no reason to also include factors where $\CI$ is replaced by $\CI'$ since these are pulled back via $p_+ \times p_S$, and hence pass through the direct image, due to the projection formula). The short exact sequence $0 \rightarrow \CI' \rightarrow \CI \rightarrow \pi_1^*(\CL) \otimes (p_S \times \text{Id})^*(\CO_\Delta) \rightarrow 0 $ on $\Hilb_{n,n+1} \times S$ yields:
$$
\pi_{1*} \left[ \prod_i \ch_{k_i} \left( \frac {\CO_{\Hilb_{n,n+1} \times S}}{\CI} \right) \cdot \pi_2^*(\gamma) \right] =
$$
$$
= \pi_{1*} \left[ \prod_i \left( \ch_{k_i} \left( \frac {\CO_{\Hilb_{n,n+1} \times S}}{\CI'} \right) - \sum_{a+b = k} \frac {\pi_1^*(c_1(\CL))^a}{a!} \cdot (p_S \times \text{Id})^*(\ch_b(\CO_\Delta)) \right) \cdot \pi_2^*(\gamma) \right]
$$
Because the Chern character of $\CO_\Delta$ equals $[\Delta]$ multiplied by a class in $R(S)$, then:
\begin{equation}
\label{eqn:2}
\pi_{1*} \left[ \prod_i \ch_{k_i} \left( \frac {\CO_{\Hilb_{n,n+1} \times S}}{\CI} \right) \cdot \pi_2^*(\gamma) \right] =
\end{equation} 
$$
= \text{sum of expressions of the form } c_1(\CL)^{a'} \cdot \pi_{1*} \left[ \prod_i \left( \ch_{k'_i} \left( \frac {\CO_{\Hilb_{n,n+1} \times S}}{\CI'} \right) \right) \cdot \pi_2^*(\gamma') \right]
$$
for various $a',k_i' \in \BN, \gamma' \in R(S)$ (above, we used $\pi_{1*}(p_S \times \text{Id})^*([\Delta]) = 1$). The right-hand side of \eqref{eqn:2} lies in $A^*_\big(\Hilb_{n,n+1})$, as expected. Similarly, we have:
\begin{equation}
\label{eqn:3}
\ch_{k'}(\Gamma^*(\CI)) = \ch_{k'}(\Gamma^*(\CI')) + \ch_{k'}(\CL \otimes (p_S)^*(\CO_\Delta|_\Delta)) = 
\end{equation}
$$
= \ch_{k'}(\Gamma^*(\CI')) + \text{sum of expressions of the form } c_1(\CL)^{a'} \cdot p_S^*(\gamma')
$$
for various $a' \in \BN$, $\gamma' \in R(S)$. Using formulas \eqref{eqn:2} and \eqref{eqn:3}, one may write \eqref{eqn:1} as a sum of products of big tautological classes on $\Hilb_{n+1} \times S$, times:
$$
(p_+ \times p_S)_*(c_1(\CL)^{d}) \stackrel{\eqref{eqn:segre 2}}= (-1)^{d} \ch_{d+2}(\CI) \qquad \text{for various } d \in \BN
$$
Therefore, we conclude that \eqref{eqn:1} is a big tautological class, hence $(p_+ \times p_S)_*$ maps $A^*_\big(...)$ to $A^*_\big(...)$. The computation that shows that $\pi_{+*}$ maps $A^*_\big(...)$ to $A^*_\big(...)$ is analogous, so we leave it as an exercise to the interested reader. \\

\noindent Let us now prove the statement about $\fq_k\fq_{k'}(\Delta)$, assuming $k,k' > 0$ (the cases when $k$ or $k'$ are negative are analogous). By \eqref{eqn:factor alpha 1}, the operator $\fq_k \fq_{k'}(\Delta)$ is given by:
$$
\xymatrixcolsep{4pc}\xymatrix{A^*(\Hilb_{n-k-k'}) \ar[r]^-{e_\downarrow \circ (e_\rightarrow)^{k'-1} \circ e^\uparrow} & A^*(\Hilb_{n-k} \times S) \ar[r]^-{e_\downarrow \circ (e_\rightarrow)^{k-1} \circ e^\uparrow} & A^*(\Hilb_{n} \times S \times S) \ar[dll]_-{\text{Id}_{\Hilb_n} \times \Delta^*} \\ A^*(\Hilb_{n} \times S) \ar[rr]^{\pi_{1*}} & & A^*(\Hilb_{n})}
$$ 
Repeating the argument for $\fq_k(\gamma)$ from the previous paragraphs shows that applying the top-most two maps in the display above to any big tautological class takes it to a sum of products of the following types of classes on $\Hilb_{n} \times S \times S$: \\

\begin{itemize}
	
\item pull-backs of classes \eqref{eqn:classes} from $A^*(\Hilb_{n})$ \\

\item pull-backs of classes in $R(S \times S) \subset A^*(S \times S)$ \\

\item the Chern classes of the universal ideal sheaves $\CI_1$ and $\CI_2$, which are pulled back from either of the two projections $\Hilb_n \times S  \times S \rightarrow \Hilb_n \times S$ \\

\end{itemize}

\noindent When we restrict the classes above to the diagonal $\Delta : S \hookrightarrow S \times S$, we simply obtain a big tautological class on $\Hilb_n \times S$, as defined in \eqref{eqn:fo}. Pushing forward such a class to $\Hilb_n$ via the first projection lands in the subring generated by big tautological classes (by the very definition of the latter), as was needed to prove.

\end{proof}

\subsection{}

Recall the tautological line bundle $\CL$ on $\Hilb_{n,n+1}$, and the operator of multiplication by its first Chern class:
\begin{equation}
\label{eqn:tau}
\fr : \bigoplus_{n=0}^\infty A^*(\Hilb_{n,n+1}) \xrightarrow{\cdot c_1(\CL)} \bigoplus_{n=0}^\infty A^*(\Hilb_{n,n+1})
\end{equation}
Consider the following operators, analogous to those of Theorem \ref{thm:nakajima} (the maps $p_{\pm}, p_S, \pi_\pm$ were defined in \eqref{eqn:diagram z1} and \eqref{eqn:diagram z2}):
\begin{align*}
&e_\downarrow^{(1)} : \bigoplus_{n=0}^\infty A^*(\Hilb_{n,n+1}) \longrightarrow \bigoplus_{n=0}^\infty A^*(\Hilb_{n+1} \times S), \qquad e_\downarrow^{(1)} = (p_+ \times p_S)_* \circ \fr \\
&e_\rightarrow^{(1)} : \bigoplus_{n=0}^\infty A^*(\Hilb_{n,n+1}) \longrightarrow \bigoplus_{n=0}^\infty A^*(\Hilb_{n+1,n+2}), \ \qquad e_\rightarrow^{(1)} = \pi_{+*} \pi^*_- \circ \fr
\end{align*}
as well as the analogous notation for the $f$ operators. In the following formulas, we will often refer to $t$ as a class on $\Hilb_{n,n+1}$, explicitly given by $p_S^*(c_1(\CK_S))$. \\

\begin{proposition}
\label{prop:aux 1}

We have the following equalities:
\begin{equation}
\label{eqn:aux 1}
\pi_{+*} \pi_-^* \pi_{-*} \pi_+^* + \emph{Id} = (p_- \times p_S)^* \circ (p_- \times p_S)_* 
\end{equation}
\begin{multline}
\pi_{+*} \pi_-^* \circ \fr \circ \pi_{-*} \pi_+^* + \fr = \\
= - p_S^*(t) \cdot \pi_{+*} \pi_-^* \pi_{-*} \pi_+^* +  \fr \circ (p_- \times p_S)^* \circ (p_- \times p_S)_* + (p_- \times p_S)^* \circ (p_- \times p_S)_* \circ \fr \label{eqn:aux 2}
\end{multline}
of operators $\bigoplus_{n=0}^\infty A^*(\eHilb_{n,n+1}) \rightarrow \bigoplus_{n=0}^\infty A^*(\eHilb_{n,n+1})$. \\

\end{proposition}

\noindent  Formulas \eqref{eqn:aux 1} and \eqref{eqn:aux 2} are straightforward restatements of the equalities \eqref{eqn:formula 1} and \eqref{eqn:formula 2} of cycles (for the convenience of the reader, the individual summands in \eqref{eqn:aux 1} and \eqref{eqn:aux 2} precisely match the respective summands in \eqref{eqn:formula 1} and \eqref{eqn:formula 2}, in order from left to right). This fact uses base change \eqref{eqn:base change} and the fact that the squares \eqref{eqn:fiber square 1} and \eqref{eqn:fiber square 2} are derived. \\

\begin{proposition}
\label{prop:aux 2}

For any $k \in \BN$, we have the formulas:
\begin{equation}
\label{eqn:aux 3} 
\sum_{i+j = k}^{i,j > 0} \fq_i \fq_j \Big|_\Delta = e^{(1)}_\downarrow \circ \underbrace{e_\rightarrow \circ ... \circ e_\rightarrow}_{k-1 \text{ operators}} \circ e^\uparrow - e_\downarrow \circ \underbrace{e_\rightarrow \circ ... \circ e^{(1)}_\rightarrow}_{k-1 \text{ operators}} \circ e^\uparrow - t(k-1) \fq_k
\end{equation}
and:
\begin{equation}
\label{eqn:aux 4} 
\sum_{i+j = k}^{i,j > 0} \fq_{-i} \fq_{-j} \Big|_\Delta =  f_\downarrow \circ \underbrace{f_\leftarrow \circ ... \circ f^{(1)}_\leftarrow}_{k-1 \text{ operators}} \circ f^\uparrow - f^{(1)}_\downarrow \circ \underbrace{f_\leftarrow \circ ... \circ f_\leftarrow}_{k-1 \text{ operators}} \circ f^\uparrow - t(k-1) \fq_{-k}
\end{equation}
where we recall that $t = c_1(\CK_S) \in A^*(S)$ multiplies $\fq_k : A^*(\eHilb) \rightarrow A^*(\eHilb \times S)$ by multiplying the $S$ factor. \\

\end{proposition}

\begin{proof} We will only prove \eqref{eqn:aux 3}, as \eqref{eqn:aux 4} is analogous. By \eqref{eqn:factor alpha 1}, we have:
\begin{equation}
\label{eqn:boo}
\fq_i \fq_j \Big|_\Delta = e_\downarrow \circ \underbrace{e_\rightarrow \circ ... \circ e_\rightarrow}_{i-1 \text{ operators}} \circ e^\uparrow \circ e_\downarrow \circ \underbrace{e_\rightarrow \circ ... \circ e_\rightarrow}_{j-1 \text{ operators}} \circ e^\uparrow \Big|_\Delta
\end{equation}
Formula \eqref{eqn:excess} translates into the identity $e_\rightarrow \circ e^\uparrow \circ e_\downarrow = e_\rightarrow^{(1)} \circ e_\rightarrow - e_\rightarrow \circ e_\rightarrow^{(1)} - t e_\rightarrow \circ e_\rightarrow$, and therefore the right-hand side of \eqref{eqn:boo} equals:
$$
e_\downarrow \circ \underbrace{e_\rightarrow \circ ... \circ e_\rightarrow}_{i-1 \text{ operators}} \circ e^\uparrow \circ e_\downarrow \circ \underbrace{e_\rightarrow \circ ... \circ e_\rightarrow}_{j-1 \text{ operators}} \circ e^\uparrow = e_\downarrow \circ \underbrace{e_\rightarrow \circ ... \circ e^{(1)}_\rightarrow}_{i-1 \text{ operators}} \circ e_\rightarrow \circ \underbrace{e_\rightarrow \circ ... \circ e_\rightarrow}_{j-1 \text{ operators}} \circ e^\uparrow - 
$$
$$
- e_\downarrow \circ \underbrace{e_\rightarrow \circ ... \circ e_\rightarrow}_{i-1 \text{ operators}} \circ e_\rightarrow^{(1)} \circ \underbrace{e_\rightarrow \circ ... \circ e_\rightarrow}_{j-1 \text{ operators}} \circ e^\uparrow - t e_\downarrow \circ \underbrace{e_\rightarrow \circ ... \circ e_\rightarrow}_{i-1 \text{ operators}} \circ e_\rightarrow \circ \underbrace{e_\rightarrow \circ ... \circ e_\rightarrow}_{j-1 \text{ operators}} \circ e^\uparrow
$$
Restricting to $\Delta$ and summing the right-hand sides over all $i+j = k$ yields \eqref{eqn:aux 3}. 

\end{proof}

\begin{proof} \emph{of Theorem \ref{thm:lehn}:} We will first prove \eqref{eqn:degree}. Fix $n \in \BN$, and for any $k \in \{0,...,n\}$ denote by $S_k$ the composition below (notation as in \eqref{eqn:diagram z1} and \eqref{eqn:diagram z2}):
\begin{multline*}
A^*(\Hilb_n \times S) \xleftarrow{(p_+ \times p_S)_*} A^*(\Hilb_{n-1,n}) \xleftarrow{\pi_{+*} \pi_-^{*}}  ... \xleftarrow{\pi_{+*} \pi_-^{*}} \\
\xleftarrow{\pi_{+*} \pi_-^{*}} A^*(\Hilb_{n-k-1,n-k}) \xleftarrow{\pi_{-*} \pi_+^{*}} ... \xleftarrow{\pi_{-*} \pi_+^{*}}  A^*(\Hilb_{n-1,n}) \xleftarrow{p_+^*} A^*(\Hilb_n) 
\end{multline*}
If we apply \eqref{eqn:aux 1}, we obtain for all $k \geq 1$:
$$
S_k + S_{k-1} = e_\downarrow \circ \underbrace{e_\rightarrow \circ ... \circ e_\rightarrow}_{k-1 \text{ operators}} \circ e^\uparrow \circ (-1)^k f_\downarrow \circ \underbrace{f_\leftarrow \circ ... \circ f_\leftarrow}_{k-1 \text{ operators}} \circ f^\uparrow \Big|_\Delta = 
$$
\begin{equation}
\label{eqn:s k}
= (-1)^k \fq_{k} \fq_{-k} \Big|_\Delta
\end{equation}
where the last equality combines \eqref{eqn:factor alpha 1} and \eqref{eqn:factor alpha 2}. Meanwhile, $S_0 = (p_+ \times p_S)_*p_+^*$ and the projection formula together with \eqref{eqn:segre 1} implies that $S_0$ is equal to the usual pullback map $A^*(\Hilb_n) \rightarrow A^*(\Hilb_n \times S)$ followed by:
$$
\text{multiplication by } (p_+ \times p_S)_*(1) = \text{multiplication by } c_2 (\CI \otimes \CK^{-1}_S) =
$$
\begin{equation}
\label{eqn:s 0}
= \text{multiplication by } \ch_2 \left( \CO_{\CZ} \right)
\end{equation}
Taking the alternating sum of \eqref{eqn:s k} for all $k\geq 1$ with \eqref{eqn:s 0} yields precisely \eqref{eqn:degree}. \\

\noindent Now let us prove \eqref{eqn:lehn}. Fix $n \in \BN$, and for any $k \in \{0,...,n\}$ let us denote by $A_k$, $B_k$, $C_k$ the three compositions below (notation as in \eqref{eqn:diagram z1} and \eqref{eqn:diagram z2}):
\begin{multline*}
A^*(\Hilb_n \times S) \xleftarrow{(p_+ \times p_S)_*} A^*(\Hilb_{n-1,n}) \xleftarrow{\pi_{+*} \pi_-^{*}}  ... \xleftarrow{\pi_{+*} \pi_-^{*}} A^*(\Hilb_{n-k-1,n-k}) \xleftarrow{\fr} \\
\xleftarrow{\fr} A^*(\Hilb_{n-k-1,n-k}) \xleftarrow{\pi_{-*} \pi_+^{*}} ... \xleftarrow{\pi_{-*} \pi_+^{*}}  A^*(\Hilb_{n-1,n}) \xleftarrow{p_+^*} A^*(\Hilb_n) 
\end{multline*}
\begin{multline*}
A^*(\Hilb_n \times S) \xleftarrow{(p_+ \times p_S)_*} A^*(\Hilb_{n-1,n}) \xleftarrow{\fr} A^*(\Hilb_{n-1,n}) \xleftarrow{\pi_{+*} \pi_-^{*}}  ... \xleftarrow{\pi_{+*} \pi_-^{*}} \\
\xleftarrow{\pi_{+*} \pi_-^{*}} A^*(\Hilb_{n-k-1,n-k}) \xleftarrow{\pi_{-*} \pi_+^{*}} ... \xleftarrow{\pi_{-*} \pi_+^{*}}  A^*(\Hilb_{n-1,n}) \xleftarrow{p_+^*} A^*(\Hilb_n) 
\end{multline*}
\begin{multline*}
A^*(\Hilb_n \times S) \xleftarrow{(p_+ \times p_S)_*} A^*(\Hilb_{n-1,n}) \xleftarrow{\pi_{+*} \pi_-^{*}}  ... \xleftarrow{\pi_{+*} \pi_-^{*}} A^*(\Hilb_{n-k-1,n-k}) \xleftarrow{\pi_{-*} \pi_+^{*}} \\
 \xleftarrow{\pi_{-*} \pi_+^{*}} ... \xleftarrow{\pi_{-*} \pi_+^{*}}  A^*(\Hilb_{n-1,n}) \xleftarrow{\fr} A^*(\Hilb_{n-1,n}) \xleftarrow{p_+^*} A^*(\Hilb_n) 
\end{multline*}
If we apply \eqref{eqn:aux 2}, we obtain:
\begin{multline*}
A_k + A_{k-1} = - t S_k + e_\downarrow \circ \underbrace{e_\rightarrow \circ ... \circ e^{(1)}_\rightarrow}_{k-1 \text{ operators}} \circ e^\uparrow \circ (-1)^k f_\downarrow \circ \underbrace{f_\leftarrow \circ ... \circ f_\leftarrow}_{k-1 \text{ operators}} \circ f^\uparrow + \\  + e_\downarrow \circ \underbrace{e_\rightarrow \circ ... \circ e_\rightarrow}_{k-1 \text{ operators}} \circ e^\uparrow \circ (-1)^k f^{(1)}_\downarrow \circ \underbrace{f_\leftarrow \circ ... \circ f_\leftarrow}_{k-1 \text{ operators}} \circ f^\uparrow
\end{multline*}
while if we apply \eqref{eqn:aux 1}, we have:
$$
B_k + B_{k-1} = e_\downarrow^{(1)} \circ \underbrace{e_\rightarrow \circ ... \circ e_\rightarrow}_{k-1 \text{ operators}} \circ e^\uparrow \circ (-1)^k f_\downarrow \circ \underbrace{f_\leftarrow \circ ... \circ f_\leftarrow}_{k-1 \text{ operators}} \circ f^\uparrow
$$
$$
C_k + C_{k-1} = e_\downarrow \circ \underbrace{e_\rightarrow \circ ... \circ e_\rightarrow}_{k-1 \text{ operators}} \circ e^\uparrow \circ (-1)^k f_\downarrow \circ \underbrace{f_\leftarrow \circ ... \circ f^{(1)}_\leftarrow}_{k-1 \text{ operators}} \circ f^\uparrow
$$
The three relations above, together with \eqref{eqn:aux 3} and \eqref{eqn:aux 4}, yield:
\begin{multline}
-A_k-A_{k-1}+B_k+B_{k-1}+C_k+C_{k-1} = t S_k + (-1)^k \cdot \\
\left[ \sum_{i+j = k}^{i,j > 0} \fq_i\fq_j \fq_{-k} \Big|_\Delta + t \sum_{k=1}^\infty (k-1) \fq_k \fq_{-k} \Big|_\Delta + \sum_{i+j = k}^{i,j > 0} \fq_k\fq_{-i}\fq_{-j} \Big|_\Delta + t \sum_{k=1}^\infty (k-1) \fq_k \fq_{-k} \Big|_\Delta \right] \label{eqn:may 1}
\end{multline}
Meanwhile, $-A_0 + B_0 + C_0$ is equal to $(p_+ \times p_S)_* \circ \fr \circ p_+^*$, hence the projection formula together with \eqref{eqn:segre 1} implies that:
\begin{multline}
- A_0 + B_0 + C_0 = \text{multiplication by } (p_+ \times p_S)_*(c_1(\CL)) \stackrel{\eqref{eqn:segre 1}}= \\ = - \text{multiplication by } c_3 (\CI \otimes \CK^{-1}_S) = \text{multiplication by } 2 \ch_3 \left( \CO_{\CZ} \right) - t \ch_2 \left( \CO_{\CZ} \right) \label{eqn:may 2}
\end{multline}
Since multiplication by $\ch_2(\CO_\CZ)$ is $- \sum_{k = 1}^\infty \fq_k \fq_{-k} |_\Delta$ by \eqref{eqn:degree}, we may take the alternating sum of \eqref{eqn:may 1} for all $k\geq 1$ with \eqref{eqn:may 2} and obtain:
\begin{multline*}
\text{multiplication by } 2 \ch_3 \left( \CO_{\CZ} \right) + t \sum_{k = 1}^\infty \fq_k \fq_{-k} \Big|_\Delta  = t \sum_{k=1}^\infty (-1)^{k-1} S_k - \\
- \sum_{i+j = k}^{i,j > 0} \fq_i\fq_j \fq_{-k} \Big|_\Delta - t(k-1) \fq_k \fq_{-k} \Big|_\Delta - \sum_{i+j = k}^{i,j > 0} \fq_k\fq_{-i}\fq_{-j} \Big|_\Delta - t(k-1) \fq_k \fq_{-k} \Big|_\Delta
\end{multline*}
By combining \eqref{eqn:s k} and \eqref{eqn:s 0}, we obtain $\sum_{k=1}^\infty (-1)^{k-1} S_k =  \sum_{k=1}^\infty (k-1) \fq_k\fq_{-k}|_\Delta$. Therefore, the relation above implies:
$$
\text{multiplication by } 2 \ch_3 \left( \CO_{\CZ} \right) = - \sum_{i+j = k}^{i,j > 0} \fq_i\fq_j \fq_{-k} \Big|_\Delta - \sum_{i+j = k}^{i,j > 0} \fq_k\fq_{-i}\fq_{-j} \Big|_\Delta - tk \fq_k \fq_{-k} \Big|_\Delta
$$
Dividing by $2$ yields formula \eqref{eqn:lehn}. 

\end{proof}

\end{document}